\documentclass[english,reqno]{amsart}

\usepackage{latexsym}
\usepackage{amssymb,amsbsy,amsmath,amsfonts,amssymb,amscd}
\usepackage{mathrsfs}

\usepackage{float}
\usepackage[caption = false]{subfig}
\usepackage[final]{graphicx}

\usepackage{subfig}
\usepackage{graphicx}

\usepackage{color}

\usepackage{cases}
\theoremstyle{plain}

\newtheorem{theorem}{Theorem}

\newtheorem{proposition}{Proposition}

\theoremstyle{definition} 
\newtheorem{definition}{Definition}
\newtheorem{remark}{Remark}

\renewcommand{\vec}[1]{\mathbf{#1}}

\newcommand{\average}[1]{ \langle#1 \rangle}

\newcommand{\NullL}{{\rm Null} \,\opL}

\newcommand{\Amat}{\mathsf{A}}
\newcommand{\AAmat}{\mathsf{AA}}
\newcommand{\Bmat}{\mathsf{B}}
\newcommand{\Pmat}{\mathsf{P}}
\newcommand{\Dmat}{\mathsf{D}}
\newcommand{\Xmat}{\mathsf{X}}
\newcommand{\Qmat}{\mathsf{Q}}

\newcommand{\opL}{\mathcal{L}}

\newcommand{\rd}{\mathrm{d}}

\newcommand{\RR}{\mathbb{R}}
\newcommand{\Sp}{\mathbb{S}}
\newcommand{\Kn}{\mathsf{Kn}}

\newcommand{\bottom}{\left( \frac{\sigma_{s0}}{\Kn} + \Kn \sigma_a \right)^2}

\title{Stability of stationary inverse transport equation in diffusion scaling}
\author{Ke Chen} 
\address{Mathematics Department, University of Wisconsin-Madison, 480 Lincoln Dr., Madison, WI 53705 USA.}
\email{ke@math.wisc.edu}
\author{Qin Li} 
\address{Mathematics Department, University of Wisconsin-Madison, 480 Lincoln Dr., Madison, WI 53705 USA.}
\email{qinli@math.wisc.edu}
\author{Li Wang} 
\address{Department of Mathematics, Computational and Data-Enabled Science and Engineering Program, State University of New York at Buffalo, 244 Mathematics Building, Buffalo, NY 14206 USA.}
\email{lwang46@buffalo.edu}
\thanks{All three authors thank the two anonymous referees for the very careful reading of the paper. It leads to significant improvement.
The work of K.C. and Q. L. is supported in part by a start-up fund of Q.L. from UW-Madison and National Science Foundation under the grant DMS-1619778. The work of L.W. is supported in part by a start-up fund from SUNY Buffalo and the National Science Foundation under the grant DMS-1620135.}

\begin{document}
\maketitle

\begin{abstract}
We consider the inverse problem of reconstructing the optical parameters for stationary radiative transfer equation (RTE) from velocity-averaged measurement. The RTE often contains multiple scales characterized by the magnitude of a dimensionless parameter---the Knudsen number ($\Kn$). In the diffusive scaling ($\Kn \ll 1$), the stationary RTE is well approximated by an elliptic equation in the forward setting. However, the inverse problem for the elliptic equation is acknowledged to be severely ill-posed as compared to the well-posedness of inverse transport equation, which raises the question of how uniqueness being lost as $\Kn \rightarrow 0$. We tackle this problem by examining the stability of inverse problem with varying $\Kn$. We show that, the discrepancy in two measurements is amplified in the reconstructed parameters at the order of $\Kn^p~ (p = 1\text{ or} ~2)$, and as a result lead to ill-posedness in the zero limit of $\Kn$. Our results apply to both continuous and discrete settings. Some numerical tests are performed in the end to validate these theoretical findings. 
\end{abstract}

\section{Introduction}

In this paper, we study the stability of inverse stationary radiative transfer equation (RTE) in different regimes. RTE is a stereotype kinetic equation that describes the dynamics of photon particles in materials with various optical properties~\cite{CaseZweifel}. The optical properties are characterized by two parameters---the scattering coefficient and absorption coefficient. Generally speaking, we denote $f(x,v)$ the distribution of particles at location $x$ moving with velocity $v$, and it obeys:
\begin{equation*}
v\cdot\nabla_xf = \int k(x,v,v')f(x,v')\rd{v'} - \sigma(x,v) f\,.
\end{equation*}
Here $x\in\Omega\subset\RR^d$ with $d=2,3$ depending on the dimension of the problem, and $v\in\Sp^{d-1}$, the unit sphere in $\RR^d$. It indicates that the particles move with fixed speed (unified to $1$) and therefore has one fewer dimension than $x$. $k(x,v,v')$ is termed the scattering coefficient, representing the probability of particles that move in direction $v'$ changing to direction $v$. $\rd{v}$ has the normalized unit measure. $\sigma(x,v)$ is the total absorption coefficient that represents certain amount of photon particles being absorbed and scattered by the material. The boundary condition is typically imposed as Dirichlet type. We separate the ``out-going" and ``in-coming" part of boundary by defining:
\begin{equation}\label{eqn:boundary}
\Gamma_\pm = \{(x,v): x\in\partial\Omega\,, \pm v\cdot n_x >0\}\,,
\end{equation}
where $n_x$ is the normal direction pointing out of $\Omega$ at point $x\in\partial\Omega$. In this way, $\Gamma_-$ collects all boundary coordinates that represent particles coming into the domain where $\Gamma_+$ collects the opposite. It is on the first set we impose Dirichlet boundary condition:
\begin{equation} \label{eqn:BC}
f|_{\Gamma_-} = \phi(x,v)\,.
\end{equation}
The well-posedness is summarized from~\cite{Lions93} and the solution is proved to be unique with full Dirichlet data.

Although the investigation of RTE is already enormous and generally acknowledged to be well-understood a long time ago, the inverse stationary RTE still attracted lots of attention in the past decade. The new life of RTE lies in the booming of medical techniques and a vast of medical images that require mathematical interpretation. In diffuse optical tomography, for example, near infra-red light (NIR) are sent into biological tissues, and by measuring the outgoing photon current at the surfaces of the tissues, scientists expect to ``invert" the problem for the optical properties of the tissue. Mathematically that means one adjusts the incoming data $f|_{\Gamma_-}$, and measure a certain form of the out-going data $f|_{\Gamma_+}$ seeking for $k(x,v,v')$ and $\sigma(x,v)$. Technically it is equivalent to seek for scattering coefficient $k(x,v,v')$ and absorption coefficient defined as 
\begin{equation}
\sigma_a(x,v)=\sigma(x,v)-\int_{\mathbb{S}^{n-1}}k(x,v,v')\rd{v'} \,,
\end{equation}
which represents the probability of particles being absorbed by the material only. A vast literature has addressed the problem from various perspectives. The well-posedness of the inverse problem in a generic setting with $\sigma_a(x)$ independent of $v$ was addressed in a pioneering paper~\cite{ChoulliStefenov}, and the uniqueness based on gauge-invarience when $\sigma_a(x,v)$ presents $v$ dependence was shown in~\cite{Stefenov2}. The idea was to decompose the albedo operator according to the singularities. Another approach is to linerize the equation before applying inverse Born series, and show the convergence of the series~\cite{Machida_Schotland_inverse}. The results on the stability of the ``inverse" dates back to~\cite{Wang} and was made systematic in~\cite{Bal09, Bal10a, Bal08}. Many papers concern the time-dependent case and the associated stability analysis has also been conducted~\cite{Larsen88, Romanov96, CS96}, and also \cite{Bal_review} for a review.

Aside from the analytical studies, various numerical techniques are explored accordingly. Numerical treatments could be separated into two categories depending on whether linearization is conducted, both on the original physical domain, or via Green's function representation~\cite{Machida:16,Machida_Schotland_inverse}. Either way, the resulting numerical problem is typically not well-posed: it is either under-determined or over-determined. The ill-posedness may be inherited from the continuous problem or due, in part, to the lack of data in experiments. The latter reason induces a purely numerical problem that can often be handled via optimization along with some regularization technique. Examples include the standard $L_2$ regularization~\cite{SRH05} for the smallness, TV regularization~\cite{Tang} for the least variance, $H_1$ norm for some regularity, and $L_1$ regularization~\cite{Ren_review} for sparsity, or Tikhonov type on each element in the inverse Born series~\cite{Machida_Schotland_inverse}. The optimization techniques are borrowed accordingly, and efficiency and memory cost for both the Jacobian-type method and the Gradient based method have been compared. The way to set the regularization coefficient, on the other hand, is usually guided by the tolerance on the error~\cite{EHN96, Egger2015} and convergence speed.

One very interesting phenomenon associated with inverse stationary transport equation is its connection to the diffusion limit. It has been longly known in the area~\cite{RenBalHielscher_trans_diff, Arridge99, Arridge98} that the diffusion approximation could serve as a substitute under certain scaling, and instead of inverting the transport equation, one studies the Calder\'on-like problems. That is, recovering the diffusion and attenuation coefficients $\frac{1}{\sigma_s}$ and $\sigma_a$ in the following equation
\begin{equation}
-C\nabla_x\cdot \left(\frac{1}{\sigma_s}\nabla_xu \right) + \sigma_a u = 0, \qquad x \in \Omega
\end{equation}
using the Dirichlet boundary condition $u|_{\partial\Omega}$ and the measurement of flux at boundaries $\frac{\partial u}{\partial n}|_{\partial \Omega}$. Here $\sigma_s$ and $\sigma_a$ are parameters derived from $\sigma$ and $k$ from the RTE equation and $C$ is a constant that only depends on the dimension. The ill-posedness of the Calder\'on problem has been shown~\cite{Uhlmann03,Uhlmann09} and its failure in capturing RTE based phenomenon in some scenario has been demonstrated~\cite{Arridge98}, or in a similar problem on diffusion approximation in recovering the doping profile in the Boltzmann-Poisson system~\cite{Gamba}.

Despite the popularity of both problems, to the best knowledge of the authors, there has not been much study on exploring the connections between the two~\cite{Arridge_Schotland09}. Specifically, some questions need to be addressed, such as: when and to what extent can diffusion approximation be used for RTE-based inverse problem? Will such approximation affect the stability in the inverse problem? Considering one type is ill-posed while the other one is well-posed, what is lost when such approximation is performed? In this paper, we make a first attempt to tackle these questions. In particular, we adopt a linearized framework (detailed in the next section), and study the well-posedness and the stability issues when passing to the diffusion limit for three different scenarios: recovering absorption coefficient, and recovering scattering coefficient in both critical and subcritical cases. We would like to mention that there are works on the change of stability with respect to certain parameters in the equation. In~\cite{time_harmonic,BalMonard_time_harmonic} the authors particularly studied the stability of the inversion with respect to the modulation frequency in time-harmonic setting, and found that the increasing of the frequency brings more details in the recovery. In~\cite{Uhlmann_stability_acoustic} the authors studied the stability of recovering acoustic equation.

We emphasize that the current paper concentrates on the illposedness as the transport equation approaches the diffusion regime (losing the stability). Except one special example in 1D (Theorem 5), we {\it assume} injectivity in our results. This is not a bizarre assumption as injectivity is shown for the associated nonlinear version of the problem \cite{ChoulliStefenov}. Indeed, injectivity and stability are two {\it separate} issues in inverse problem: the former one concerns the uniqueness in recovery whereas the latter measures the accuracy in recovery when small perturbation in measurement is allowed. In terms of the spectral theory, the injectivity requires that the spectrum is away from zero, and stability studies the whole span of the spectrum. While stability naturally being the next step after injectivity, it is not uncommon that stability can be studied by assuming injectivity. As pointed out in \cite{Bal_review}, stability is studied in the setting of isotropic source and angularly averaged measurements while the uniqueness is not available \cite{Bal08,time_harmonic,BalMonard_time_harmonic}.   

The rest of paper is organized as follows. We present some preliminaries in the next section, including the derivation of the diffusion equation from the RTE and the set-up of the inverse problem in full generality. Section 3 and 4 are devoted to the three scenarios described above respectively. In all three cases, we utilize the linearization approach, study the well-posedness of the problem in both regimes, and examine the change of stability while passing to the diffusion limit. We also introduce a distinguishability parameter to indicate the stability and we justify that the inverse problem becomes more and more indistinguishable in the diffusion limit. Numerical tests are exploited to demonstrate the statements on the properties.

\section{Preliminaries}
Some preliminaries are collected in this section. The first subsection demonstrates the derivation of the diffusion equation from the RTE in the forward problem and the second subsection sets up the inverse problem we study in a general framework. For the conciseness of the paper we assume $\sigma_s(x)$ and $\sigma_a(x)$ do not have $v$ dependence. 
\subsection{Diffusion limit}
The diffusion limit sets in when scattering is strong and absorption is weak. The equation in the dimensionless form reads:
\begin{equation}\label{eqn:RTE}
\begin{cases}
v\cdot\nabla_xf = \frac{1}{\Kn}\sigma_s\mathcal{L}f - \Kn\sigma_af\,,\\
f|_{\Gamma_-} = \phi(x,v)\,,
\end{cases}
\end{equation}
where $\mathcal{L}$ is the collision operator and in the velocity independent case it writes as:
\begin{equation}\label{eqn:collision}
\mathcal{L} f = \int f(x,v')\rd{v'} - f = \langle f\rangle - f\,.
\end{equation}
Here we have re-grouped the gain term and loss term in the collision for the ease of later presentation. A more general collision takes the form $\mathcal{L} f = \int k(v',v)(f(x,v')-f(x,v)) \rd{v}'$, but our analysis in the rest of the paper can be easily adapted to this case. We therefore keep it in the simplest form, and assume that $\sigma_s$ has no $v$ dependence. There are two key features of the collision operator:
\begin{itemize}
\item{Mass conservation}: $\int \mathcal{L}[f]\rd{v} = 0$. If we apply this property to the original equation, immediately we see that $\int vf\rd{v}$ is a divergence free field if $\sigma_a = 0$;
\item{One dimensional Null space}: By setting $\mathcal{L}[f]=0$, one gets $f = \langle f\rangle$, meaning that $f$ is a constant in velocity domain. We denote it as $\NullL = \{\rho(x)\}$, the collection of functions that depend on $x$ only. This property is unique for RTE compared to other linear kinetic equation and it is the main reason that the asymptotic limit only consist of a scalar equation instead of a system.  
\end{itemize}

As $\Kn\to 0$, the equation falls into the diffusion limit and we have the following theorem:
\begin{theorem}\label{thm:diffusion}
Suppose $f$ solves~\eqref{eqn:RTE}. As $\Kn\to 0$, $f(x,v)$ converges to $\rho(x)$, which solves the diffusion equation:
\begin{equation}\label{eqn:diff}
\begin{cases}
C\nabla_x\cdot \left(\frac{1}{\sigma_s}\nabla_x\rho \right) - \sigma_a\rho= 0\,,\\
\rho|_{\partial\Omega} = \xi_f\,,
\end{cases}
\end{equation} 
Here $C$ is a constant depending on the dimension of the problem. The boundary condition is determined by:
\begin{equation*}
\xi_f(x_0) = f^l_{z\to\infty}\,.
\end{equation*}
with $f^l$ solving
\begin{equation*}
\begin{cases}
v_z\partial_zf^l = \sigma_s\mathcal{L}[f^l]\,,  \quad z\in [0, \infty) \\
f^l|_{z=0} = \phi(x_0,v)\,.
\end{cases}
\end{equation*}
\end{theorem}
\begin{proof}
The proof follows the standard asymptotic expansion with boundary layer analysis. In the zero limit of $\Kn$, the distribution in the interior will stabilize whereas the boundary condition $\phi$ being away from the equilibrium function will prevent the solution converging near the boundary. To separate the two, we first expand the solution:
\begin{equation}
f(x,v) = f_\text{bd} + f_\text{in}\,,
\end{equation}
where $f_\text{bd}$ is the solution adjacent to boundary accounting for the boundary layer, while $f_\text{in}$ characterizes the interior away from the layer. We study $f_\text{in}$ first. As $\Kn \rightarrow 0$, we apply the standard asymptotic expansion technique and write:
\begin{equation}
f_\text{in} = f_0 + \Kn f_1 + \Kn^2 f_2 +\cdots\,.
\end{equation}
Here we only consider the expansion away from the boundary layer so that $f_\text{bd}$ is negligible. Inserting the expansion in the equation~\eqref{eqn:RTE} and equate like powers of $\Kn$:
\begin{itemize}
\item[$\mathcal{O}(1)$] $\mathcal{L}f_0 = 0$. This immediately indicates that $f_0\in\NullL$. With the form given in~\eqref{eqn:collision}, $\NullL$ consists functions that are constants in $v$ domain, and thus $f_0(x,v)=\rho(x)$.
\item[$\mathcal{O}(\Kn)$] $v\cdot\nabla_xf_0 = \sigma_s\mathcal{L}[f_1]$. This indicates that $f_1 = \mathcal{L}^{-1}\left(v\cdot\nabla_xf_0\right)$. $\mathcal{L}$ is not a one-to-one map unless the domain is confined in $\NullL^\perp$, and the inverse on $\mathcal{L}$ is pseudo-inverse. Considering the form of $\mathcal{L}$ in~\eqref{eqn:collision}, then $\NullL^\perp = \{f:\int f\rd{v} = 0\}$, and thus $f_1 = -\frac{v}{\sigma_s}\cdot\nabla_x\rho$.
\item[$\mathcal{O}(\Kn^2)$] $v\cdot\nabla_xf_1 = \sigma_s\mathcal{L}[f_2] - \sigma_af_0$. Here we integrate the equation with respect to $v$. The second term will vanish and the left hand side becomes:
\begin{equation}\label{eqn:kn_zero_limit}
\int v\cdot\nabla_x\left(-\frac{v}{\sigma_s}\cdot\nabla_x\rho\right) \rd{v}= -\sigma_a\rho \quad\Rightarrow\quad C\nabla_x\cdot\left(\frac{1}{\sigma_s}\nabla_x\rho\right) = \sigma_a\rho\,.
\end{equation}
Here the constant $C$ depends on the dimension of the velocity space. 
\end{itemize}
Summarizing up the analysis we obtain $f_\text{in}(x,v)\to f_0(x,v) = \rho(x)$ that solves the diffusion equation \eqref{eqn:diff}.

We then need to provide $\rho$ a correct boundary condition and this comes from the treatment of $f_\text{bd}$. For this we follow ~\cite{BLP}. At each point $x\in\partial\Omega$, we perform tangential approximation and by stretching coordinates we can locally change the problem into a half space problem. More specifically, let $x_0\in\partial\Omega$ and $n_x$ the normal direction pointing out of the domain, we denote $z = -\frac{n_x\cdot (x-x_0)}{\Kn}$. For every fixed point $x = x_0 - n_x y (y\in[0,\infty))$ away from the layer along the ray pointing into the domain, $z\to\infty$ as $\Kn\to 0$. After moving the coordinate frame to $x_0$ with $n_x$ direction, $z=0$ stands for boundary point and $z=\infty$ is mapped to the interior. Then along $z$ direction the equation reads, in the leading order of $\Kn$:
\begin{equation}\label{eqn:half_space}
\begin{cases}
v_z\partial_zf = \sigma_s\mathcal{L}[f]\,, \quad z \in [0, \infty) \\
f|_{z=0} = \phi(x_0,v)\,.
\end{cases}
\end{equation}
It is a half space problem in $z$ with boundary condition given only at $z=0$. The problem is proved to have a unique solution and is computed in~\cite{LLS} and the infinite data on $z$ will be a constant in $v$ direction, which will be used to serve as the Dirichlet boundary condition for $\rho$, meaning:
\begin{equation*}
\rho(x_0) = f_{z\to\infty}\,.
\end{equation*}
This is done at each grid point along the boundary and we end up with the boundary condition for $x\in \partial \Omega$, denoted by $\xi_f[\phi](x)$ (or $\xi_f(x)$ for short).
\end{proof}

\begin{remark}
The proof here is formal and is not specific on certain norm. In fact with general geometry and boundary condition it is believed to be correct but not proved yet. In 2-D physical domain and 1-D velocity sphere, due to the joint force of~\cite{Guo,LLS_geometry}, it can be made rigorous in $L_\infty(\rd{x}\rd{v})$ norm when boundary layer is excluded, meaning that as $\Kn\to 0$:
\begin{equation}
\|f-\rho\|_{L_\infty(\rd{x_i}\rd{v})} = \mathcal{O}(\Kn^{2/3})\to 0\,,
\end{equation}
and the decay from $\mathcal{O}(\Kn)$ to $\mathcal{O}(\Kn^{2/3})$ is mainly due to the curvature correction, which is controllably small but nontrivial. Here $L_\infty(\rd{x_i}\rd{v})$ stands for $L_\infty$ norm in the interior only. We exclude a fixed small boundary layer of $\mathcal{O}(\Kn)$ width. 
\end{remark}

\begin{remark}
Very frequently, RTE in 3D is simplified under some symmetry assumption. Specifically, we assume it with slab plan geometry. Denote $x= (x_1,x_2,x_3)\in \Omega =  \mathbb{R}^2\times (0,1)$ and $v=(\sin\theta\cos\phi,\sin\theta\sin\phi,\cos\theta) \in \mathcal{S}^2$, then we assume that all parameters and conditions are homogenized along $x_3$ direction, i.e. $\sigma_a(x)=\sigma_a(x_3)$, $\sigma_s(x)=\sigma_s(x_3)$, and that boundary condition is $\phi(x,v) = \phi(x_3,\cos\theta)$. The stationary RTE becomes:
\begin{equation}
\begin{aligned}
v\cdot\nabla_x f(x,v) =
\frac{\sigma_s(x_3)}{\Kn}\mathcal{L}f(x,v) -\Kn\sigma_a(x_3) f(x,v) 
\end{aligned}
\end{equation}
with:
\begin{equation*}
\mathcal{L}f(x,v) =\frac{1}{4\pi}\int_0^\pi \int_0^{2\pi} [f(x,v')-f(x,v)] \sin\theta'\rd{\phi'}\rd{\theta'}\,,
\end{equation*}
and the boundary condition $f|_{\Gamma_-}=\phi(x_3,\cos\theta)$. With $v\cdot \nabla_x = \cos\theta \partial_{x_3}$, the stationary RTE becomes:
\begin{equation}
\begin{cases}
\cos\theta \partial_{x_3}f(x_3,\cos\theta)=\frac{\sigma_s(x_3)}{\Kn}(\frac{1}{4\pi}\int_0^\pi\int_0^{2\pi}( f(x_3,\cos\theta')-f(x_3,\cos\theta)) \sin\theta'\rd{\phi'}\rd{\theta'} )-\Kn\sigma_a(x_3) f(x_3,\cos\theta) \,, \\
f|_{\Gamma_-} = \phi(x_3,\cos\theta) \,.
\end{cases}
\end{equation}
For simplicity, we denote $x=x_3\in (0,1)$ and make change of variable $v=\cos\theta\in (-1,1)$ to obtain:
\begin{equation}\label{eqn:1DRTE}
\begin{cases}
v\partial_x f(x,v) = \frac{\sigma_s(x)}{\Kn} \int_{-1}^1 f(x,v')-f(x,v)\frac{\rd{v'}}{2} -\Kn\sigma_a(x)f(x,v)\,, \quad (x,v)\in (0,1)\times [-1,1] \,,\\
f|_{\Gamma^-} = \phi(x,v)\,,
\end{cases}
\end{equation}
where $v$ is often termed as direction of flight. 
Sending $\Kn\to 0$ and follow the same asymptotic derivation in the theorem above, one obtains $C = \frac{1}{3}$ in the zero Knudsen number limit in~\eqref{eqn:kn_zero_limit}. From here on, we always refer to equation \eqref{eqn:1DRTE} as the 1-D RTE. Notice here the velocity domain is interval $[-1,1]$. 
\end{remark}

\subsection{Inverse problem}
The inverse problem can be set up associated with a map from the input data on one portion of the boundary to the measurement on the other portion. For RTE specifically, this map is often termed the albedo operator. Depending on the data-acquisition method in the experiments, the measurement can take various forms. Here we assume that only velocity-averaged measurement is available and define the measurement operator:
\begin{equation}
\mathcal{M}f(x) = \int_{\Gamma_+(x)} v\cdot n(x) f(x,v)\rd{v}\,,
\end{equation}
where $\Gamma_+(x)$ is the ``outgoing'' semisphere at $x$ defined in \eqref{eqn:boundary}. And the albedo operator reads:
\begin{equation}\label{eqn:map}
\mathcal{A}_{\Kn}(\sigma_a,\sigma_s)\,:\quad\phi |_{\Gamma_-}\rightarrow\mathcal{M}f\,,
\end{equation}
where $\phi$ is the Dirichlet boundary condition \eqref{eqn:BC} and $\mathcal{M}f(x)$ is the intensity of light propagating out of the domain at boundary point $x\in \partial\Omega$. Then the inverse problem is to recover $\sigma_a$ and $\sigma_s$ given the information of the map $\mathcal{A}_\Kn$.

On the theoretical level, one concerns about the well-posedness and stability. The well-posedness problem states the following: given the full information on the map $\mathcal{A}$, can one {\it uniquely} recover $\sigma_a$ and $\sigma_s$? The answer is positive if a velocity-resolved measurement (i.e., $f|_{\Gamma_+}$ for any $v$ instead of \eqref{eqn:map}) or time-dependent measurement is available \cite{ChoulliStefenov, Stefenov2}. The analysis is based on singularity separation. Since the albedo operator is a forward map, an explicit form can be obtained and it consists of three parts, separated according to their singular level: the most singular part is a delta function that could be used to recover $\sigma_a$ through the inverse X-ray transform, and the secondly singular term is used to recover $\sigma_s$, leaving the third term in $L_\infty$. The stability problem, on the other hand, asks: if the entire map is off from the accurate one by a small amount of error ($\|\mathcal{A} - \tilde{\mathcal{A}}\|<\epsilon$), how accurate the recovering could be? That is, will $\|\sigma_a - \tilde{\sigma}_a\|$ remain small? The problem is examined in \cite{Bal10a, Bal09, Bal08} using the same kind of singular decomposition.

On the numerical level, the well-posedness problem takes a slightly different form: let $\sigma_{s,a}$ be discretized at $N_x$ grid points, then if for each incoming data $\phi_d$, one could take $N_p$ measurements at the boundary, how many incoming data is needed to fully recover $\sigma_{s,a}$? This is intrinsically the same as in the theoretical level. The stability problem, however, has two sides. One is, assume that the data obtained is exact, how much error one obtains in recovering $\sigma_{s,a}$ if $\Amat$ gets perturbed a little. Here $\Amat$ is the matrix representation of $\mathcal{A}$. This amounts to analyzing the condition number of the problem. Another is that if the measurement in~\eqref{eqn:map} is off by some error, how accurate will $\sigma_{s,a}$ be? The answer to that lies in analyzing the norm of $\Amat^{-1}$. The first question is aligned with theoretical stability stated above, while the second question is a pure numerical issue. Concisely, since solving the inverse problem numerically often involves inverting a matrix (or a series of matrices in the nonlinear setting), even if the problem is well-posed, the numerical error is still hard to guarantee. All these questions are indirectly or partially resolved in many papers~\cite{Arridge99, Ren_review}.

We ask a different question in this paper. Our aim is to investigate the dependence of stability and condition number for the inverse problem on the Knudsen number. It is generally acknowledged that the inverse transport problem is well-posed (for most kinds of measurements and under mild assumptions on the scattering/absorption coefficients) whereas the inverse diffusion equation as a limit is ill-posed. Considering the two sets of equations are connected by simply passing to the $\Kn\to 0$ limit, is there a more explicit explanation of the loss of uniqueness? 

To answer this question, we adopt a linearization framework \cite{Ren_review}. By assuming that the to-be-recovered coefficients $\sigma_a$ and $\sigma_s$ are only slightly deviated from some given functions from {\it a priori} knowledge, we linearize the transport equation around them and write down the relationship between the unknown parameters and the map $\mathcal{A}$. As a result and as will be more clear later, we only need to invert a Fredholm operator of the first kind to recover those parameters. More importantly, this Fredholm operator reveals an explicit dependence on the Knudsen number, which allow us to analyze the stability of the inversion when varying the magnitude of $\Kn$. We would also like to point out that the measurement we took in \eqref{eqn:map} contains minimum information that no uniqueness results on the reconstruction based on the singular decomposition are available. Nevertheless, the linearization approach we take allows us to quantify the stability explicitly in $\Kn$, and is amenable for numerical schemes. Details will be provided along the paper. 

As discussed before, depending on the assumptions on the media, different scenarios could take place. We present here three examples. The following sections are respectively devoted to recovering $\sigma_a$ when $\sigma_s$ is known, and recovering $\sigma_s$ in the critical ($\sigma_a = 0$) and subcritical ($\sigma_a >0$) cases, respectively. It is not our aim in this paper to give a complete analysis for all scenarios but rather to investigate the problem for the first time and nail down the techniques that could make it possible.

\section{Recover Absorption Coefficient $\sigma_a$}

\subsection{Inverse problem set-up}
In this section, we assume that the scattering coefficient is known (and we set it as $\sigma_s = 1$ for simplicity), and recover the absorption coefficient $\sigma_a$. The equation \eqref{eqn:RTE} reads:
\begin{equation}\label{eqn:RTE_abs}
\begin{cases}
v \cdot \nabla_x f&=\frac{1}{\Kn}\mathcal{L}f-\Kn\sigma_af, \qquad (x, v) \in \Omega \times \mathbb{S}\,,  \\
f|_{\Gamma_-}&=\phi. \hspace{2.7cm}  
\end{cases}\,
\end{equation}
Here $\phi$ is the inflow boundary condition and the solution is denoted by $f(x,v;\phi)$. Experimentally suppose one could measure data at one grid point $y\in\partial\Omega$, then the mapping becomes:
\begin{equation*}
\mathcal{A}(\sigma_a)\,:\quad \phi \to \mathcal{M}f(\cdot)\,.
\end{equation*}

We follow a linearization framework \cite{Ren_review} and set a background absorbing coefficient $\sigma_{a0}(x)$ by assuming that the residue 
\begin{equation*}
\tilde{\sigma}_a(x) = \sigma_a(x) - \sigma_{a0}(x)\,
\end{equation*}
is much smaller than $\sigma_a$: $\|\tilde{\sigma}_a\| \ll \|\sigma_a\|$. Then the linearized problem with the same inflow boundary condition reads as
\begin{equation} \label{eqn:RTE_abs_0}
\begin{cases}
v\cdot \nabla_x f_0=\frac{1}{\Kn}\mathcal{L}f_0-\Kn\sigma_{a0}f_0\,, \\
f_0|_{\Gamma_-}=\phi\,,
\end{cases}\,
\end{equation}
where $f_0(x,v; \phi)$ is the function we linearize upon. Let
\begin{equation*}
\tilde{f}(x,v) = f(x,v) - f_0(x,v)
\end{equation*}
be the fluctuation, then it solves the following equation by subtracting \eqref{eqn:RTE_abs_0} from \eqref{eqn:RTE_abs}
\begin{equation}\label{eqn:RTE_abs_tilde}
\begin{cases}
v\cdot \nabla_x \tilde{f}=\frac{1}{\Kn}\mathcal{L}\tilde{f}-\Kn\sigma_{a0}\tilde{f}-\Kn\tilde{\sigma}_af_0\,, \\
\tilde{f}|_{\Gamma_-}=0\,,
\end{cases}
\end{equation}
where we have omitted the higher order terms. The incoming boundary information $\phi$ is implicitly contained in $f_0$. To make use of the boundary condition and measurement, we write the adjoint problem of \eqref{eqn:RTE_abs_0} and assign it a Delta function boundary condition:  
\begin{equation}\label{eqn:RTE_abs_g}
\begin{cases}
-v\cdot \nabla_x g=\frac{1}{\Kn}\mathcal{L}g-\Kn\sigma_{a0}g\,, \\
g|_{\Gamma_+}=\delta_y(x)\,.
\end{cases}
\end{equation}
Multiply the above equation by $\tilde{f}$ and subtract it from the product of \eqref{eqn:RTE_abs_tilde} and $g$, and integrate in both $x$ and $v$, we get
\begin{equation} \label{eqn:RTE_abs_IBP}
 \int  v\cdot n \tilde{f} g |_{\Gamma_+}\rd{v} \rd{x}= \Kn \int_\Omega \tilde{\sigma}_a\int f_0g\rd{v}\rd{x}\,,
\end{equation}
which defines a linear mapping from $\phi$ and $\delta_y$ with the solutions of \eqref{eqn:RTE_abs_0} \eqref{eqn:RTE_abs_g} to the measured data. 
Note that the LHS of \eqref{eqn:RTE_abs_IBP} could be easily obtained by subtracting the computed flux of $f_0$ from the measurement, i.e.,
\begin{eqnarray*}
b(\delta_y, \phi) := \int  v\cdot n \tilde{f} g |_{\Gamma_+}\rd{v} \rd{x}   =\mathcal{M}f(y)-\mathcal{M}f_0(y)  \,,
\end{eqnarray*}
and the RHS is a Fredholm operator of the first kind with known $\int f_0 g \rd{v}$ and unknown $\tilde{\sigma}_a$. Denote 
\begin{equation} \label{eqn:gamma00}
\gamma_\Kn(x;\delta_y,\phi):= \Kn \int f_0(x,v;\phi) g(x,v; \delta_y)\rd{v}\,,
\end{equation}
then the map $\mathcal{A}(\tilde{\sigma}_a):~ \phi \rightarrow  b(\delta_y, \phi)$ rewrites as 
\begin{equation}\label{eqn:LS_abs}
 \int_{\Omega} \gamma_\Kn(x;\delta_y,\phi)\tilde{\sigma}_a(x)\rd{x}=b(\delta_y,\phi)\,.
\end{equation}
Now it amounts to invert the First type Fredholm integral to recover $\tilde{\sigma}_a$.

\subsection{Ill-conditioning in the diffusion limit (continuous level)}
We study the stability of the inverse problem in terms of the Knudsen number in this subsection. More specifically, consider the linear mapping 
\begin{equation}\label{eqn:map_abs}
\langle \gamma_\Kn, \tilde{\sigma}_a \rangle_{L^2(\rd{x})}= b(\delta_y, \phi )\,,\quad\text{with}\quad\gamma_\Kn = \Kn \int f_0g\rd{v} \,, 
\end{equation}
where $\gamma_\Kn $ is defined in \eqref{eqn:gamma00} with $f_0$ and $g$ solving \eqref{eqn:RTE_abs_0} and \eqref{eqn:RTE_abs_g}, respectively. We aim to understand the influence in recovering $ \tilde{\sigma}_a$ if a small purturbation in $b(\delta_y, \phi)$ is introduced. In particular, we would like to check such sensitivity's dependence on the Knudsen number $\Kn$. To see this, we first define a distinguishability coefficient to measure the ``condition number" for a given error $\delta$ on data.

\begin{definition}
Consider linear mapping in~\eqref{eqn:map_abs}, we define the distinguishability coefficient as
\begin{equation}\label{eqn:distinguishability}
\kappa_a =\sup_{\sigma_a \in\Gamma_\delta}\frac{\| \sigma_a - \tilde{\sigma}_a\|_{L^\infty(\rd{x})}}{\| \tilde{\sigma}_a\|_{L_\infty(\rd{x})}}\,,
\end{equation}
where
\[
\Gamma_\delta=\{ \sigma_a: \sup_{\substack{\forall \|\phi\|_{L^\infty(\Gamma_-)}\leq 1,\\ \forall y\in \partial\Omega}}|\langle\gamma_\Kn\,,\sigma_a\rangle_{L^2(\rd{x})} - b(\delta_y,\phi) | \leq \delta  \}\,,
\]
and $\tilde{\sigma}_a$ is the solution to~\eqref{eqn:map_abs}.
\end{definition}
As written, $\Gamma_\delta$ is the collection of all possible solutions to the map given that the measurement is contaminated by $\delta$ error. Noticed that both $\langle \gamma_{\Kn},\sigma_a \rangle_{L^2(\rd{x})}$ and $b(\delta_y,\phi)$ are linearly dependent on $\phi$, we take the sup-norm to normalize in the definition of $\Gamma_\delta$. Then $\kappa_a$ measures the relative error that could be seen in the recovery --- smaller $\kappa_a$ leads to better distinguishability. Recall here that the stability defined in \cite{Bal_review} says if the two measurement are distinguished by $\delta$, i.e., $\| \mathcal{A} \sigma_a - \mathcal{A} \tilde{\sigma}_a\| \leq \delta $, then the discrepancy in parameters can be bounded as $\| \sigma_a - \tilde{\sigma}_a\| \leq \| \mathcal{A}^{-1}\| \| \mathcal{A}(\sigma_a - \tilde{\sigma}_a)\| = \|\mathcal{A}^{-1}\| \delta$, then $\kappa_a$ defined in \eqref{eqn:distinguishability} is seen as an estimate of $\|\mathcal{A}^{-1}\|$.

We expect to show two things: 1) smaller error tolerance $\delta$ results in better distinguishability; 2) smaller $\Kn$ drives the problem into the diffusion limit, leading to worse distinguishability. The following theorem groups the two things together.

\begin{theorem}\label{thm:Dis}
We study the inverse problem of recovering $\tilde{\sigma}_a$ in ~\eqref{eqn:map_abs}. Assume the map from $\tilde{\sigma}_a$ to $b$ is injective, given an error tolerance $\delta$ on the measurement, then the distinguishability coefficient grows as $\delta$ grows and $\Kn$ shrinks in the sense that there exists a constant $C$ such that

\begin{equation}\label{eqn:Dis}
\kappa_a \geq C\frac{\delta}{\Kn^2} ,\quad \text{when} \quad \Kn\ll 1\,.
\end{equation}

\end{theorem}

\begin{proof}
Let $c(x)$ be an arbitrary function such that vanishes in the boundary layer and 
\begin{equation}\label{eqn:map_diff}
|\langle \gamma_\Kn,c\rangle_{L^2(\rd{x})}|\leq \delta\,,
\end{equation}
where $\gamma_\Kn $ is defined in \eqref{eqn:gamma00}. Then for such a $c(x)$, one certainly has
\[
\sigma_a(x)=\tilde{\sigma}_a(x)+c(x)\in \Gamma_\delta\,.
\]
According to Theorem~\ref{thm:diffusion}, when $\Kn$ is small, $f_0$ and $g$ approach the diffusion limit. Taking boundary layers into account, we write
\begin{equation*}
f_0=f_L+f_I\quad g=g_L+g_I\,,
\end{equation*}
where $f_L$ and $g_L$ stand for the layers of the two functions, and $f_I$ and $g_I$ are the interior solutions. The boundary layers are supported in a thin layer (denoted as $\Omega_L$) in the vicinity of the boundary $\partial\Omega$ with $\mathcal{O}(\Kn)$ width. The interior $f_I$ and $g_I$ supported on $\Omega_I = \Omega \backslash \Omega_L$ are well-approximated by the diffusion limit:
\begin{equation*}
f_I=\rho_f - \Kn v \cdot \nabla \rho_f +\mathcal{O}(\Kn^2 )\,,\quad g_I=\rho_g- \Kn v \cdot \nabla \rho_g +\mathcal{O}(\Kn^2)\,,
\end{equation*}
where $\rho_f,\rho_g$ solve the following equations:
\begin{equation} \label{eqn:rho-f-0}
\begin{cases}
& C\Delta _x \rho_f  = \sigma_{a0} \rho_f, 
\\
&\rho_f |_{\partial \Omega} = \xi_f(x)\,,
\end{cases}
\end{equation}
and
\begin{equation} \label{eqn:g-0}
\begin{cases}
&C \Delta_x \rho_g  = \sigma_{a0} \rho_g, 
\\
&\rho_g |_{\partial \Omega} =\xi_g(x)\,.
\end{cases}
\end{equation}
Therefore,
\begin{equation*}
\gamma_\Kn |_{\Omega_I}= \Kn \int  f_I g_I\rd{v} = \Kn  \rho_f\rho_g + \mathcal{O}(\Kn^3)\,.
\end{equation*}
Plugging it back into~\eqref{eqn:map_diff} and using the fact $c(x)=0,\forall x\in \Omega_L$, we obtain:
\begin{eqnarray}
\langle \gamma_\Kn,c\rangle_{L^2(\rd{x})}&=&\int_{\Omega_I} \gamma_\Kn (x)c(x)\rd x+\int_{\Omega_L}\gamma_\Kn(x)c(x)\rd x \nonumber
\\ &\sim&  \Kn \int_{\Omega_I}\rho_f(x)\rho_g(x) c(x)\rd{x}+\mathcal{O}(\Kn^3) \,. \label{eqn:207}
\end{eqnarray}
Let $G(x,y)$ be the Green's function for the operator $C\Delta- \sigma_{a0}$, i.e., 
\begin{equation} \label{eqn:Green}
\begin{cases}
&C \Delta_y G - \sigma_{a0} G = \delta_x(y), 
\\
&G|_{\partial \Omega} = 0\,.
\end{cases}
\end{equation}
Then
\begin{equation} \label{eqn:205}
\rho_f(x) =  \int_{\partial \Omega} \xi_f(y) \frac{\partial G}{\partial n} (x,y) \rd \mu(y),
\quad
\rho_g(x) =  \int_{\partial \Omega} \xi_g(y) \frac{\partial G}{\partial n}(x,y) \rd\mu(y) \,,
\end{equation}
where $\rd\mu(y)$ is the surface measure on $\partial \Omega$. Note that one can find $c(x)$ such that 
\begin{equation}\label{eqn:206}
\int_{\Omega_I} \rho_f(x) \rho_g(x) c(x) \rd{x} = \mathcal{O}(\Kn)\,.
\end{equation}
Indeed, using quadrature rule for \eqref{eqn:205}, one has
\begin{equation*}
\rho_f(x) = \sum_j \frac{\partial G(x,y_j)}{\partial n} \xi_f(y_j) w_j + \mathcal{O}(\Kn), \quad 
\rho_g(x) = \sum_j \frac{\partial G(x,y_j)}{\partial n} \xi_g(y_j) w_j + \mathcal{O}(\Kn)\,.
\end{equation*}
Thus
\begin{equation} \label{eqn:orthogonal1}
\int \rho_f \rho_g c \rd{x} = \sum_{i,j} \xi_f(y_i) \xi_g(y_j) w_i w_j \int \frac{\partial G(x,y_i)}{\partial n} \frac{\partial G(x,y_j)}{\partial n} c(x) \rd{x} +\|c\|_{L^\infty(\rd{x})} \mathcal{O}(\Kn)\,,
\end{equation}
then one just need to choose $c(x)$ such that it is perpendicular to the space spanned by $\{ \frac{\partial G(x,y_i)}{\partial n} \frac{\partial G(x,y_j)}{\partial n}, ~ \forall  i, j\}$. Plugging \eqref{eqn:206} into \eqref{eqn:207}, we derive that
\begin{equation}
\langle\gamma,c\rangle_{L^2(\rd{x})}\sim \|c\|_{L^\infty(\rd{x})} \mathcal{O}(\Kn^2)+\mathcal{O}(\Kn^3)\,.
\end{equation} 
It is seen that to ensure $|\langle\gamma,c\rangle_{L^2(\rd{x})}|<\delta$, one simply needs
\begin{equation*}
\|c\|_{L^\infty(\rd{x})} \sim\mathcal{O}\left(\frac{\delta}{\Kn^2}\right)\,.
\end{equation*}

Considering $\tilde{\sigma}_a$ is of $\mathcal{O}(1)$, there exists a constant $C$ such that $\kappa_a\geq C\frac{\delta}{\Kn^2}$, we finish the proof.

\end{proof}

\begin{remark}
Two immediate take-away information from the theorem:
\begin{itemize}
\item When $\delta$ is small, meaning that the measurement is relatively accurate, then one gets better recovery of the absorption coefficient as expected.
\item When $\Kn$ shrinks, the system approaches to the diffusion limit, and the distinguishability coefficient grows dramatically, indicating that the  linearized inverse problem is highly ill-conditioned. This phenomena is aligned with the ill-posedness of the Calder\'on problem. 
\end{itemize}
\end{remark}

\subsection{Ill-conditioning in the diffusion limit (discrete level)}
In this subsection, we revisit the above observation in the discrete setting when solving the inverse problem numerically, in which case the matrix to be inverted becomes highly ill-conditioned as $\Kn$ shrinks. Let us sample $N_x$ quadrature points in $\Omega$: $x_i,i=1,\ldots,N_x$, and their corresponding weights are denoted by $w_i$. Suppose the measurements are collected at discrete boundary points $\{y_k\}_{k=1}^{N_y}\subset \partial\Omega$, and there are $\phi_d, d = 1, 2, \cdots, N_\phi$ different sets of incoming data where $\sup_d \|\phi_d\|_{L_\infty(\Gamma^-)}\leq 1$. Then the linear system~\eqref{eqn:LS_abs} can be rewritten into the form of
\begin{equation}\label{eqn:LS_abs_dis}
\forall d = 1,\cdots,N_{\phi}, \quad k=1,\cdots, N_y: \quad \sum_{i=1}^{N_x} \gamma_\Kn(x_i;y_k,\phi_d)\tilde{\sigma}_a(x_i) w_i = b(y_k,\phi_d)\,,
\end{equation}
with
\begin{equation*}
\gamma_\Kn(x_i;y_k,\phi_d) =\Kn \int f_0(x_i, v; \phi_d)g(x_i, v; \delta_{y_k})\rd{v} \,,
\end{equation*}
and $f_0$ and $g$ solve~\eqref{eqn:RTE_abs_0} and~\eqref{eqn:RTE_abs_g} respectively with $\phi_d$ and $\delta_{y_k}$ as boundary condition, here $\delta_{y_k}$ is hat function centered at $x=y_k$. For simplicity we use index $p$ to denote sub-index $k$ and $d$. Then the linear system (\ref{eqn:LS_abs_dis}) can be further written into a compact form:
\begin{equation}\label{eqn:LS_abs_dis2}
\Amat \vec{ \tilde{\sigma}_a}=\vec{b}\,,
\end{equation}
where $\Amat \in \mathbb{R}^{N_p \times N_x}$ with entries $\Amat_{pi}=\gamma_\Kn(x_i; y_k, \phi_d)w_i$ such that $(k,d) = p$. $\vec{\tilde{\sigma}_a}\in \mathbb{R}^{N_x}$ is the discretization of $\tilde{\sigma}_a$ on $N_x$ quadrature nodes, and data $\vec{b}\in\mathbb{R}^{N_p}$. Here $N_p=N_y\times N_{\phi}$ denotes the total number of data points we have, and it is a product of $N_{\phi}$, the number of inflow data and $N_y$, the number of measurement positions. We show below that as $\Kn\to 0$ the matrix $\Amat$ becomes more and more singular, making the inversion impossible.
\begin{theorem} \label{theorem:abs}
Assume that $\Amat$ is nonsingular. When $\Kn$ is small,  the condition number of matrix $\Amat^T\Amat$ scales as

\begin{equation}
\kappa(\Amat^T \Amat) \geq C\frac{1}{\Kn} \,, \quad \text{for some constant} \quad C\,.
\end{equation}

Moreover, in 1D, $\Amat$ is approximately low rank, in the sense that it only has no more than 3 singular values of size $\mathcal{O}(\Kn)$, and all the rest are of size $\mathcal{O}(\Kn^{3/2})$.
 \end{theorem}
\begin{proof}
We first the prove the theorem in any dimension and refine the result in 1D. According to the diffusion theory, in the zero limit of $\Kn$, $\Amat$ can be decomposed into two parts:
\begin{equation} \label{eqn:A-decomp}
\Amat= \Kn \begin{pmatrix}
\Amat_L &\Amat_I
\end{pmatrix}\,,
\end{equation}
where $\Amat_L \in \mathbb{R}^{N_p \times N_L}, ~ (N_L \ll N_x)$ are the sampled points in the layer, and $\Amat_I\in \mathbb{R}^{N_p  \times (N_x-N_L)}$ represents the sampled points in the interior. To analyze the rank of the matrix, we rewrite $\Amat^T\Amat$ as
\[
\Amat^T\Amat= \Kn^2
\begin{pmatrix}
\Amat_L^T\Amat_L & \Amat^T_L\Amat_I \\
\Amat_I^T\Amat_L & \Amat^T_I\Amat_I \\
\end{pmatrix}\,.
 \] 
When $\Kn$ is small, $\Amat_I=\Amat^0_I+\mathcal{O}(\Kn)$, where each row in $\Amat_I^0$ are:
\begin{equation} \label{eqn:Aij}
  \rho_f(x_i; \phi_d)\rho_g(x_i; y_k) w_i, \qquad  i= 1, 2, \cdots N_x\,.
 \end{equation}
Therefore, $\Amat^T\Amat=\Kn^2 \Pmat+\mathcal{O}(\Kn^3)$, where
\[
\Pmat=
\begin{pmatrix}
\Amat_L^T\Amat_L & \Amat^T_L\Amat_I \\
\Amat_I^T\Amat_L & \Amat^{0T}_I\Amat^0_I 
\end{pmatrix}
=
\begin{pmatrix}
0 & 0\\
0& \Amat^{0T}_I\Amat_I^{0}
\end{pmatrix}
+
\begin{pmatrix}
\Amat_L^T\Amat_L & \Amat^T_L\Amat_I\\
\Amat_I^T\Amat_L & 0\\
\end{pmatrix}\,.
\]
Since $\Amat_I^{0T}\Amat_I^0$ is symmetric, it is diagonizable. Denote:
\begin{equation}
\Amat_I^{0T}\Amat_I^0 = \Qmat\Dmat\Qmat^{-1}\,,
\end{equation}
with $\Qmat\in \mathbb{R}^{(N_x-N_L)\times (N_x-N_L)}$ the collection of eigenvectors and $\Dmat$ the diagonal matrix of eigenvalues. Then we multiply both sides by matrix $\Xmat$ and $\Xmat^{-1}$ defined as below
\[
\Xmat=
\begin{pmatrix}
\mathbb{I} & 0\\
0& \Qmat\\
\end{pmatrix}
,\quad
\Xmat^{-1}=
\begin{pmatrix}
\mathbb{I} & 0\\
0& \Qmat^{-1}\\
\end{pmatrix}\,,
\]
we derive that
\[
\Xmat^{-1}\Amat^{T}\Amat\Xmat= \Kn^2
\begin{pmatrix}
0& 0\\
0& \Dmat\\
\end{pmatrix}
+
\Kn^2\begin{pmatrix}
\Amat_L^T\Amat_L & \Amat^T_L\Amat_I\Qmat\\
\Qmat^{-1}\Amat_I^T\Amat_L & 0  \\
\end{pmatrix}\, + \mathcal{O}(\Kn^3)\,.
\]
Note that the number of elements in $A_L$ is of order $\mathcal{O}(\Kn)$ due to the fact that the layer length is order of $\mathcal{O}(\Kn)$, and its elements are order $\mathcal{O}(1)$ thanks to the maximal principle: the integrand function $|f_Lg_L|\leq \|\phi\|_{L^\infty(\Gamma_-)}\|\delta_{y_k}\|_{L^\infty(\rd{x})}=\mathcal{O}(1)$.
 Therefore, we could rewrite the equation above:
\[
\Xmat^{-1}\Amat^{T}\Amat\Xmat= \Kn^2
\begin{pmatrix}
0& 0\\
0& \Dmat\\
\end{pmatrix}
+\mathcal{O}(\Kn^3) := \Kn^2 \Dmat_1 + \mathcal{O}(\Kn^3)\,.
\]
By Gershgorin circle theorem, all eigenvalues lie in Gershgorin disc, meaning that $|\lambda(\Amat^T\Amat)-\lambda(\Kn^2 \Dmat_1)|<\mathcal{O}(\Kn^3)$. 
Since the eigenvalues in $\Dmat_1$ can be $\mathcal{O}(1)$ or 0, the largest eigenvalue in $\Amat^T\Amat$ is $\mathcal{O}(\Kn^2)$ and the smallest is $\mathcal{O}(\Kn^3)$, the condition number of $\Amat^T \Amat$ is larger than $\mathcal{O}\left( \frac{1}{\Kn}\right)$. 

Moreover, in 1D, we show in the appendix that $\Dmat$ is indeed a low rank matrix itself with rank not exceeding 3. Therefore, there are at most 3 nonzero eigenvalues in $\Dmat$ and all the rest are zeros. Putting this information back to $\Amat$, the result directly follows. 


\end{proof}

\begin{remark}
We emphasize the difference between Theorem~\ref{thm:Dis} and Theorem~\ref{theorem:abs}. Theorem~\ref{thm:Dis} is the study of $\kappa_a$, which represents $\|\mathcal{A}^{-1}\|$: it tells that suppose the measurement is different, how different could $\sigma_a$ be. However, numerically it is the condition number of $\Amat$ that is playing the role. If $\Amat$ is low rank, for example, rank $k$ out of an $n$ dimensional space, then there are infinite many $\sigma_a$ that could lead to the same measurement and the space they span is $n-k$ dimensional.
\end{remark}

Similar to the analysis on the continuous level, we immediately conclude that the ill-conditioned matrix $\Amat$ in $\Kn\to0$ limit leads to the fact that $\sigma_a$ is hard to be recovered accurately, which is consistent with the ill-posedness of the Calder\'on problem. More precisely, we have the following estimate theorem. 

\begin{theorem}
Define the distinguishability coefficient in the discrete setting as
\begin{equation*}
\kappa_a = \sup_{\sigma_a \in \Gamma_\delta} \frac{\| \sigma_a - \tilde{\sigma}_a \|}{ \|\tilde{\sigma}_a \|}\,,
\end{equation*}
where $\|\cdot\|$ denote vector $l^2$-norm and $\Amat \tilde{\sigma}_a = b$, and $\Gamma_\delta = \{ \sigma_a: \| \Amat \sigma_a - b \| \leq \delta \}$. Assume that $\Amat$ is nonsingular, then there exists a constant $C$ such that

\begin{equation} \label{eqn:222}
 \kappa_a \geq  C \frac{\delta}{ \Kn^2 + \Kn \Delta x}\,.
 \end{equation}
 
 \end{theorem}

\begin{proof}
We again decompose $\Amat$ as in \eqref{eqn:A-decomp}. Note that for each row of $\Amat_I$, we have
\begin{equation*}
(\Amat_I)_i = \int f_0(x_i, v; \phi_d) g(x_i, v; y_k) \rd{v}  w_i = \rho_f(x_i; \phi_d) \rho_g(x_i; y_k) w_i + \mathcal{O}(\Kn^2)\,,
\end{equation*}
where
\begin{eqnarray*}
&&\rho_f(x_i; \phi_d) = \sum_j \frac{\partial G(x_i, y_j)}{\partial n} \xi_{\phi_d}(y_j) w_j= \int \frac{\partial G(x_i, y)}{\partial n } \xi_{\phi_d}(y) \rd{y} + \mathcal{O}(\Delta y ) \,, \label{eqn:rhof-d}
\\ &&\rho_g(x_i; y_k) = \sum_j \frac{\partial G(x_i, y_j)}{\partial n} \xi_{\delta_{y_k}}(y_j) w_j= \int \frac{\partial G(x_i, y)}{\partial n } \xi_{\delta_{y_k}}(y) \rd{y} + \mathcal{O}(\Delta y ) \,. \label{eqn:rhog-d}
\end{eqnarray*} 
Here $G(x,y)$ is the Green's function defined in \eqref{eqn:Green} and $\Delta y = \max_i \Delta y _i$. Denote $\vec{c} = \vec{\sigma}_a - \vec{\tilde{\sigma}}_a$, then 
\begin{equation} \label{eqn:209}
\|\Amat\vec{c}\| \leq \delta\,. 
\end{equation}
Since
\begin{equation*}
 \Amat \vec{c} =  \Kn  \Amat_L \vec{c}_L   + \Kn   \Amat_I \vec{c}_I = \mathcal{O}(\Kn^2)  \vec{c}_L  + \mathcal{\Kn}   \Amat_I \vec{c}_I \,,
\end{equation*}
and from \eqref{eqn:rhof-d}, one can always choose $\vec{c}_I$ as long as $N_\phi \times N_y < N_x$ (just pick $c_I$ from \eqref{eqn:orthogonal1} ) such that 
\[
\sum_{i=1}^{N_x} \rho_f(x_i; \phi_d) \rho_g(x_i; y_k) \vec{c}_I(x_i) w_i = \mathcal{O}(\Delta y), \quad \forall \phi_d, ~ y_k\,,
\]
we have 
\begin{equation*}
 \Amat \vec{c} = \mathcal{O}(\Kn^2) \vec{c}_L + \mathcal{O}(\Kn \Delta y) \vec{c}_I\,.
\end{equation*}
Therefore, the requirement \eqref{eqn:209} implies \eqref{eqn:222}.
\end{proof}

\section{Recover Scattering Coefficient $\sigma_s$}
In this section we discuss how to recover the scattering coefficient given $\sigma_a$. In subsection~\ref{sec:sigma_s_set_up} we set up the inverse problem, and the following two subsections are devoted to the non-injectivity in 1D and ill-conditioning in multi-dimension in the zero limit of the Knudsen number.

\subsection{Inverse problem set-up}\label{sec:sigma_s_set_up}
We recall the equation again:
\begin{equation}\label{eqn:RTE_sca}
\begin{cases}
v\cdot \nabla_x f =  \frac{1}{\Kn}\sigma_s \mathcal{L} f - \Kn \sigma_a f \,,\quad &(x,v)\in\Omega\times\mathbb{S}\,,\\
 f |_{\Gamma_-}=\phi\,. &  \\
\end{cases}
\end{equation}
Here $\sigma_a$ as known, and we make a guess for $\sigma_s$, and linearize the equation around $\sigma_{s0}$, assuming the deviation $|\tilde{\sigma}_s|:=|\sigma_s-\sigma_{s0}|$ is much smaller than $\sigma_s$. The background solution $f_0$ solves
\begin{equation}\label{eqn:RTE_sca_0}
\begin{cases}
v \cdot \nabla_x f_0 = \frac{1}{\Kn} \sigma_{s0} \opL f_0 -\Kn \sigma_a f_0 \,,
\\ f_0 |_{\Gamma_-} = \phi\,.
\end{cases}
\end{equation}
As done in the last section, we drop the higher order terms, and the fluctuation $\tilde{f}:=f-f_0$ satisfies the following equation:
\begin{equation}\label{eqn:RTE_sca_tilde}
\begin{cases}
v\cdot\nabla_x \tilde{f} = \frac{1}{\Kn}\sigma_{s0} \mathcal{L} \tilde{f}+ \frac{1}{\Kn}\tilde{\sigma}_s\mathcal{L}f_0 - \Kn \sigma_a \tilde{f}\,,\\
\tilde{f}|_{\Gamma_-}=0\,. \\
\end{cases}
\end{equation}
To recover $\tilde{\sigma}_s$, we look for the linear mapping from the incoming information $\phi$ to some computable quantity $b(\delta_y, \phi)$ (to be determined below), i.e.,
\[
 \phi \rightarrow b(\delta_y, \phi)\,.
\]
To this end, we consider an auxiliary function $g(x,v)$ that satisfies the adjoint problem:
\begin{equation}\label{eqn:RTE_sca_g}
\begin{cases}
-v\cdot\nabla_x g= \frac{1}{\Kn}\sigma_{s0} \mathcal{L}g - \Kn \sigma_a g \,, \\
g|_{\Gamma_+}=\delta_{y}(x) \,. \\
\end{cases}\,
\end{equation}
Multiply Equation~\eqref{eqn:RTE_sca_tilde} with $g$ and~\eqref{eqn:RTE_sca_g} with $\tilde{f}$ and compare these two equations, with the Green's identity, one gets:
\begin{equation}\label{eqn:RTE_sca_IBP}
\int_{\Gamma_+(y)} v\cdot n(y) \tilde{f}(y,v;\phi)\rd v=\frac{1}{\Kn} \int \tilde{\sigma}_s(x)\int g(x,v; \delta_y)\mathcal{L}f_0(x,v;\phi) \rd v\rd x\,.
\end{equation}
Then we define:
\begin{equation}\label{eqn:gamma-sca-crit}
\gamma_\Kn(x; \delta_y,\phi):= \frac{1}{\Kn}\int g(x,v; \delta_y)\mathcal{L}f_0(x,v;\phi)\rd{v}\,,
\end{equation}
then \eqref{eqn:RTE_sca_IBP} becomes:
\begin{equation}\label{eqn:A_critical}
 \int \gamma_\Kn (x;\delta_y,\phi)\tilde{\sigma}_s(x) \rd{x}=b(\delta_y,\phi)\,,
\end{equation}
where
\begin{equation*}
b(\delta_y, \phi) = \mathcal{M}(f-f_0)(y;\phi)\,,
\end{equation*}
which is again the difference between the measured data and computed data, given by the boundary condition $\phi$, evaluated at $y$. The inverse problem then is equivalent to the Fredholm first type problem: for all $y$ and $\phi$, we prepare $\gamma(x;\delta_y,\phi)$ and $b(\delta_y,\phi)$, and use them to invert for $\sigma_s$. For the ease of notation we write $b(y,\phi)$ as $b(\delta_y,\phi)$ with $\delta_y$ representing the boundary condition for $g$.

\subsection{Non-injectivity in 1D}
Similar to the case of recovering $\sigma_a$, recovering $\sigma_s$ becomes harder as $\Kn$ shrinks to zero. In 1D, formulae can be made explicitly.

We first restrict our attention to the critical case by setting $\sigma_a \equiv 0$. When there is no absorption, the only interaction between particles is scattering, and thus mass is preserved. We show below that in this case, the problem is non-injective in the sense that $\gamma$ cannot provide enough variations to distinguish $\sigma_s$ at different $x$, and that different $\sigma_s$ could lead to the same measurement provided the same data. 
 
\begin{proposition}\label{thm:RTE_sca}
Given arbitrary $\phi_d$ and $\delta_y$, let $f_0$ solve~\eqref{eqn:RTE_sca_0} and $g$ solve~\eqref{eqn:RTE_sca_g}. Let $\gamma_\Kn$ be defined in~\eqref{eqn:gamma-sca-crit}. Then in 1D, if $\sigma_a \equiv 0$, $\gamma_\Kn$ is a constant independent of $x$.
\end{proposition}
\begin{proof}
Here we drop the $\Kn$ dependence in the proof as it will change the result. Denote $\langle f\rangle=\frac{1}{2}\int_{-1}^1 f(x,v)\rd{v}$, then
\begin{equation*}
\frac{\rd}{\rd x}\gamma_\Kn=\frac{1}{\Kn}\frac{\rd}{\rd x}(\langle f_0\rangle \langle g\rangle-\langle f_0g\rangle ) \,.
\end{equation*}
Notice that 
\begin{equation*}
g=\langle g\rangle +\frac{1}{\sigma_{s0}}v \partial_x g\,,\qquad
f_0 =\langle f_0\rangle-\frac{1}{\sigma_{s0}}v \partial_x f_0\,,
\end{equation*}
thus
\begin{align*}
\frac{\rd}{\rd x}\langle f_0 g\rangle&=\langle \partial_x f_0 g\rangle+\langle f_0\partial_x g\rangle\\
&=\langle \partial_xf_0  \rangle\langle g\rangle +\frac{1}{\sigma_{s0}}\langle \partial_xf_0 v\partial_x g\rangle \\
&\quad +\langle f_0\rangle \langle \partial_xg\rangle -\frac{1}{\sigma_{s0}}\langle \partial_xf_0 v\partial_x g\rangle \\
&=\frac{\rd}{\rd x}(\langle f_0\rangle\langle g\rangle)\,,
\end{align*}
which readily implies that $\frac{\rd}{\rd x}\gamma_\Kn=0$\,.

\end{proof}
The non-injectivity is immediate:
\begin{theorem}
In 1D critical case, the inverse problem for~\eqref{eqn:A_critical} is non-injective.
\end{theorem}
\begin{proof}
According to proposition~\ref{thm:RTE_sca}, $\gamma_\Kn$ is a constant, and thus one can only get $\int\tilde{\sigma}_s\rd{x}$. Therefore, the variation in $\tilde{\sigma}_s$ is not recoverable, and the problem is non-injective.
\end{proof}
 
In the subcritical case with $\sigma_a >0$, the recovery of the scattering coefficient is very similar.
\begin{proposition}\label{prop:RTE_sca2}
For arbitrary inflow data $\phi$ and Dirac delta function $\delta_y$, as the Knudsen number $\Kn \rightarrow 0$, $\frac{\rd}{\rd x}\gamma_\Kn \rightarrow 0$.
\end{proposition}
\begin{proof}
Notice that
\[
\begin{cases}
g&=\frac{v}{\Kn \sigma_a}\partial_x g+\frac{\sigma_{s0}}{\Kn^2\sigma_a}\mathcal{L}g\\
f_0&=-\frac{v}{\Kn \sigma_a} \partial_x f_0+\frac{\sigma_{s0}}{ \Kn^2 \sigma_a}\mathcal{L}f_0\\
\end{cases}
\]
we can derive that
\[
\begin{aligned}
\frac{\rd}{\rd x}\langle f_0 g\rangle&=\langle g\partial_xf_0\rangle+\langle f_0\partial_x g\rangle\\
&=\langle \partial_xf_0\frac{v}{\Kn\sigma_a}\partial_xg\rangle +\langle \partial_xf_0\rangle \frac{\sigma_{s0}}{\Kn^2 \sigma_a}\langle g\rangle -\langle g\partial_xf_0\rangle \frac{\sigma_{s0}}{\Kn^2\sigma_a}\\
&\quad +\langle -\partial_xf_0\frac{v}{\Kn\sigma_a}\partial_xg\rangle +\langle f_0\rangle \frac{\sigma_{s0}}{\Kn^2 \sigma_a}\langle \partial_x g\rangle -\langle f_0\partial_xg\rangle \frac{\sigma_{s0}}{\Kn^2 \sigma_a}\\
&=\frac{\sigma_{s0}}{\Kn^2\sigma_a}\frac{\rd}{\rd x}(\langle f_0\rangle \langle g\rangle)-\frac{\sigma_{s0}}{\Kn^2\sigma_a}\frac{\rd}{\rd x}\langle f_0g\rangle\\
\end{aligned}
\]
this is equivalent to
\[
\frac{\rd}{\rd x}\langle f_0 g\rangle =\frac{\sigma_{s0}/\sigma_a}{\Kn^2+\sigma_{s0}/\sigma_a}\frac{\rd}{\rd x}(\langle f_0 \rangle\langle g\rangle)\,,
\]
therefore we have
\[
\frac{\rd}{\rd x}\gamma_\Kn =\frac{\Kn}{\Kn^2+\sigma_{s0}/\sigma_a}\frac{\rd}{\rd x}(\langle f_0\rangle \langle g\rangle) \rightarrow 0 \qquad \text{ as } ~\Kn \rightarrow 0\,.
\]
\end{proof}

This proposition indicates that as $\Kn\to 0$, whatever boundary condition we provide for $f_0$ and $g$, $\gamma_\Kn$ is going to be flat, and thus not able to distinguish the variation in $\sigma_s$, making the problem more and more non-injective.

\begin{theorem}\label{thm:RTE_sca2}
When Kundsen number $\Kn \rightarrow 0$, solving linear system $\langle\gamma_\Kn,\tilde{\sigma}_s\rangle_{L^2(\rd{x})}=b$ is non-injective. More specifically, in the zero limit of $\Kn$, the space spaned by the kernel $\gamma_\Kn$ has finite rank, impossible to reflect full information of $\tilde{\sigma}_s$.
\end{theorem}
\begin{proof}
The non-injectivity of the linear system directly follows from Proposition~\ref{prop:RTE_sca2}. Moreover, we show in below that the space that $\gamma$ resides in is of low rank in the zero limit of $\Kn$. Let the domain of $x$ be $x_l \leq x \leq x_r$ and denote
\[
\langle f_0\rangle:= \rho_f, \quad \langle g \rangle := \rho_g\,.
\] 
As Knudsen number $\Kn$ goes to zero, $ \rho_f$ and $ \rho_g$ solves the diffusion equation in the leading order:
\[
\begin{aligned}
\frac{\rd}{\rd x} \left(\frac{1}{\sigma_{s0}}\frac{d}{dx}\rho_f \right)=\sigma_a \rho_f \,,   \quad
\frac{\rd}{\rd x} \left(\frac{1}{\sigma_{s0}}\frac{d}{dx} \rho_g \right)= \sigma_a \rho_g \,,
\end{aligned}
\]
with velocity averaged boundary data $\xi_f$ and $\xi_g$ at only two points: left and right end points. That is,
\[
\rho_f(x_l) = \xi_{f_l}, ~ \rho_f(x_r) = \xi_{f_r};  \qquad \rho_g(x_l) = \xi_{g_l}, ~ \rho_f(x_r) = \xi_{g_r}\,.
\]
In 1D, there are two Green's functions:
 \[
 \frac{\partial}{\partial y} \left(\frac{1}{\sigma_{s0}}\frac{\partial }{\partial y}G_1 \right)=\sigma_a G_1\,, \quad G_1(x,y=x_l) = 1\,,\quad G_1(x,y=x_r) = 0\,,
 \]
 and
 \[
 \frac{\partial}{\partial y} \left(\frac{1}{\sigma_{s0}}\frac{\partial }{\partial y}G_2 \right)=\sigma_a G_2\,, \quad G_2(x,y=x_l) = 0\,,\quad G_2(x,y=x_r) = 1\,.
 \]
Then $\rho_f$ and $\rho_g$ can be written as
\[
\rho_f = \xi_{f_l} G_1(x) + \xi_{f_r} G_2(x); \quad \rho_g = \xi_{g_l} G_1(x) + \xi_{g_r} G_2(x)\,.
\]
As a result, $\rho_f(x)\rho_g(x) \in \text{span}\{G_1(x)^2, G_1(x)G_2(x), G_2(x)^2\}$, which means that the function space of $\rho_f \rho_g$ is of low rank (rank 3). Combining with Proposition~\ref{prop:RTE_sca2}, we see that:
\begin{equation*}
\frac{\rd}{\rd x}\gamma_\Kn = \frac{\Kn}{\Kn^2+\sigma_{s0}/\sigma_a}\frac{\rd}{\rd x}\left(\rho_f\rho_g\right)\,,
\end{equation*}
meaning $\frac{\rd}{\rd x}\gamma_\Kn$ is low rank as well. Therefore it is impossible to recover $\tilde{\sigma}_s$ from the linear mapping $\langle \gamma_\Kn,\cdot \rangle_{L^2(\rd{x})}$.
\end{proof}
\begin{remark}
Theorem \ref{thm:RTE_sca2} coincides with our intuition in Section 1 since when Knudsen number is small, the scattering will dominate absorption as neutron travels through the medium, and the case described here converges to the critical case.
\end{remark}

\subsection{Ill-conditioning in higher dimensions}
In higher dimension, we show that the inverse problem become more and more singular as the Knudsen number approaches zero. 

Recall the linear mapping in this scenario:
\begin{equation*}
\average{\gamma_{\Kn}, \tilde{\sigma}_s} =  b(\delta_y, \phi_d), \qquad \text{with} \quad 
\gamma_{\Kn} = \frac{1}{\Kn}\int \opL f_0 g \rd{v}\,, 
\end{equation*}
where $f_0$ and $g$ solve \eqref{eqn:RTE_sca_0} \eqref{eqn:RTE_sca_g}, respectively. We then investigate the sensitivity of reconstructing $\tilde{\sigma}_s$. The main theorem states as follows. 
\begin{theorem}\label{thm:Dis-crit}
Assume that the map from $\tilde{\sigma}_s$ to $b$ is injective. Given an error tolerance $\delta$ on the measurement, the distinguishability coefficient in reconstructing $\tilde{\sigma}_s$ grows as $\delta$ grows and $\Kn$ shrinks. Namely, there exists a constant $C$ such that
\begin{equation}\label{eqn:Dis}
\kappa_s : = \sup_{\sigma_s \in \Gamma_\delta} \frac{\| \sigma_s - \tilde{\sigma}_s \|_{L^\infty(\rd{x})}}{ \| \tilde{\sigma}_s \|_{L^\infty(\rd{x})}} \geq C\frac{\delta}{\Kn} ,\quad \text{when} \quad \Kn\ll 1\,,
\end{equation}
where 
\[
\Gamma_\delta =\{ \sigma: \sup_{\substack{\forall \|\phi\|_{L^\infty(\Gamma_-)}\leq 1,\\ \forall y\in \partial\Omega}} |\langle\gamma_\Kn\,,\sigma\rangle_{L^2(\rd{x})} - b(\delta_y,\phi_d)|\leq \delta  \} \,.
\]
\end{theorem}
\begin{proof}
When $\Kn$ is small, we decompose $f_0$ and $g$ into a layer part that accounts for the boundary layer supported in the vicinity of the boundary $\partial \Omega$ with $\mathcal{O}(\Kn)$ width and an interior part:
\begin{equation*}
f_0 = f_L + f_I, \quad g = g_L + g_I\,.
\end{equation*}

Decompose also $\gamma_\Kn$ as $\gamma_\Kn = \left(\gamma_\Kn \right)_L + \left( \gamma_\Kn\right)_I$, then we examine the interior part $(\gamma_\Kn)_I$ by means of $f_I$ and $g_I$. Upon asymptotic expansion, the interior solutions $f_I$ and $g_I$ are approximated as follows
\begin{eqnarray*} 
f_I &=& \rho_f - \frac{\Kn}{\sigma_{s0}} v \cdot \nabla_x \rho_f + \Kn^2f_2 \,,
\\ g_I &=& \rho_g  + \frac{\Kn}{\sigma_{s0}} v \cdot \nabla_x \rho_g + \Kn^2g_2\,,
\end{eqnarray*}
hence, 
\begin{eqnarray} \label{eqn:0223}
(\gamma_\Kn)_I &=& \frac{1}{\Kn} ( \average{f_I} \average{g_I} - \average{f_I g_I}) \nonumber
\\ &=& \frac{1}{\Kn} \left\{  (\rho_f-\Kn^2 \average{f_2}) (\rho_g + \Kn^2 \average{g_2} )   \right.  \nonumber
\\ && \qquad  \left. - \left\langle    \left( \rho_f - \frac{\Kn}{\sigma_{s0}} v\cdot \nabla_x \rho_f + \Kn^2f_2\right) 
\left(\rho_g - \frac{\Kn}{\sigma_{s0}} v\cdot \nabla_x \rho_g + \Kn^2g_2 \right) \right\rangle   \right\} \nonumber
\\ &=& - \frac{\Kn}{\sigma_{s0}^2} \average{  (v\cdot \nabla_x \rho_f ) (v \cdot \nabla_x \rho_g)}  + \mathcal{O}(\Kn^2) \nonumber
\\ & = & - \frac{\Kn}{\sigma_{s0}^2}  C \nabla_x \rho_f \cdot \nabla_x \rho_g  + \mathcal{O}(\Kn^2)\,,
\end{eqnarray}
where $C$ is a constant depending on the dimension of the problem. Now consider $c(x)$ such that it vanishes in the boundary layer and that
\begin{equation*}
|\average{\gamma_\Kn, c}_{L^2(\rd{x})}| \leq \delta\,,
\end{equation*}
then one certainly has
\[
\sigma_s(x) = c(x) +\tilde{\sigma}_s(x)\in\Gamma_\delta\,.
\]
Since $\average{\gamma_\Kn, c}_{L^2(\rd{x})} = \average{(\gamma_\Kn)_I, c_I} _{L^2(\rd{x})} = \average{ - \frac{\Kn}{\sigma_{s0}^2}  C \nabla_x \rho_f \cdot \nabla_x \rho_g  , c_I}_{L^2(\rd{x})} + \mathcal{O}(\Kn^2)$, 
we immediately have
\begin{equation*}
\|c\|_{L^\infty(\rd{x})} \sim \mathcal{O}\left( \frac{\delta }{\Kn} \right)\,,
\end{equation*}
which implies that

\[
\kappa_s \geq  C\frac{\delta }{\Kn} \,,\quad \text{for some constant}\quad C\,.
\]

\end{proof}

\begin{remark}\label{remark:0203}
Notice that in the expression \eqref{eqn:0223}, the leading term in $\gamma$ has the structure of $\nabla_x \rho_f \cdot \nabla_x \rho_g $, where $\rho_f$ and $\rho_g$ solve the diffusion equation with corresponding boundary condition, which has a Green's function formation. Therefore, one can show that it is asymptotically ``low rank" in the sense that one can find a nonzero $c(x)$ such that 
\[
\int \nabla_x \rho_f \cdot \nabla_x \rho_g c \rd{x} = \mathcal{O}(\Kn)\,.
\]
Details follow the argument in the proof of Theorem \ref{thm:Dis}\,.
\end{remark}

\section{Numerical test}
In this section, we conduct a few numerical experiments in 1D to check the conditioning of the inverse problem, and show that it degrades when $\Kn$ goes to zero, as indicated by the theory above. More specifically, we check the variation of $\gamma(x; \delta_y,\phi)$ in $x$ and examine the singular values of the matrix $\Amat$, whose element takes values $\Amat_{pi} = \gamma(x_i; \delta_{y_k},\phi_d)$ with $(k,d)=p$. 

Recall the definition of $\gamma(x; \delta_{y_k},\phi_d)$ in \eqref{eqn:gamma00} and \eqref{eqn:gamma-sca-crit}, one needs to solve the forward problems for $f_0$ and $g$ with boundary conditions $\phi_d$, and $\delta_y$, respectively. The forward solver we adopt is the fully implicit solver \cite{LW17} with Generalized minimal residual method (GMRES) under a tolerance of $10^{-10}$. For all the examples below, the spatial domain is chosen as $\Omega=[0,1]$ and discretized with $N_x=200$ uniformly distributed nodes. We also sample $N_v = 80$ grid points in the velocity domain $\mathcal{S}=[-1,1]$. Note that for the $\Kn$ we considered here, the spatial mesh is fine enough to resolve the boundary layers. 

Now we need to decide the boundary condition $f_0|_{\Gamma_-} = \phi_d(x,v)$ and $g|_{\Gamma_-} = \delta_y(x)$ such that we extract the most information from $f_0$ and $g$. Since $f_0$ solves a linear equation, we can set delta functions as its inflow data, i.e.,
\[
\phi_d(0,v) = \delta(v-v_d), \quad  v_d >0, \quad  d = 1, \cdots, 40\,;  \qquad
\phi_d(1,v) = \delta(v-v_d), \quad  v_d <0, \quad  d = 41, \cdots, 80\,.
\]
The boundary condition for $g$ is easier to set up. Note that in 1D, there are only two boundary points, $x=0$ and $x=1$, therefore, the boundary for $g$ reads
\[
\delta_{y_1} = \delta(x-0), \quad \delta_{y_2} = \delta(x-1)\,.
\]
Then the associated matrix is of size $\Amat\in \mathbb{R}^{160\times 200}$. 

\subsection{Recover Absorption Coefficient}
The first test addresses the problem of recovering $\sigma_a$. Here the background absorption is set to be $\sigma_{s}=1+\frac{1}{1.5+\sin(2\pi x)}$ and the scattering coefficient takes the form $\sigma_{a0} = 4+\frac{1}{2}\sin(4\pi x)$. Here 
\begin{equation} \label{Api}
\Amat_{pi} = \Kn \int g(x_i,v; \delta_{y_k}) f_0(x,v; \phi_d) \rd{v}, \quad p = (d,k)\,.
\end{equation}
As predicted in Theorem~\ref{theorem:abs}, in 1D, the matrix $\Amat$ is approximately low rank with rank 3, which indicates that its singular value decays to zero quickly. This is demonstrated in Fig.~\ref{fig:sin_decay_int_abs}, wherein we plot the singular values for a variety of $\Kn$. It is easy to see that only three singular values stand out, coincide with the rank 3 argument. And the difference between the first three singular values with the rest are more pronounced with smaller $\Kn$. 
\begin{figure}[htp]
\centering
  \includegraphics[width=.5\linewidth]{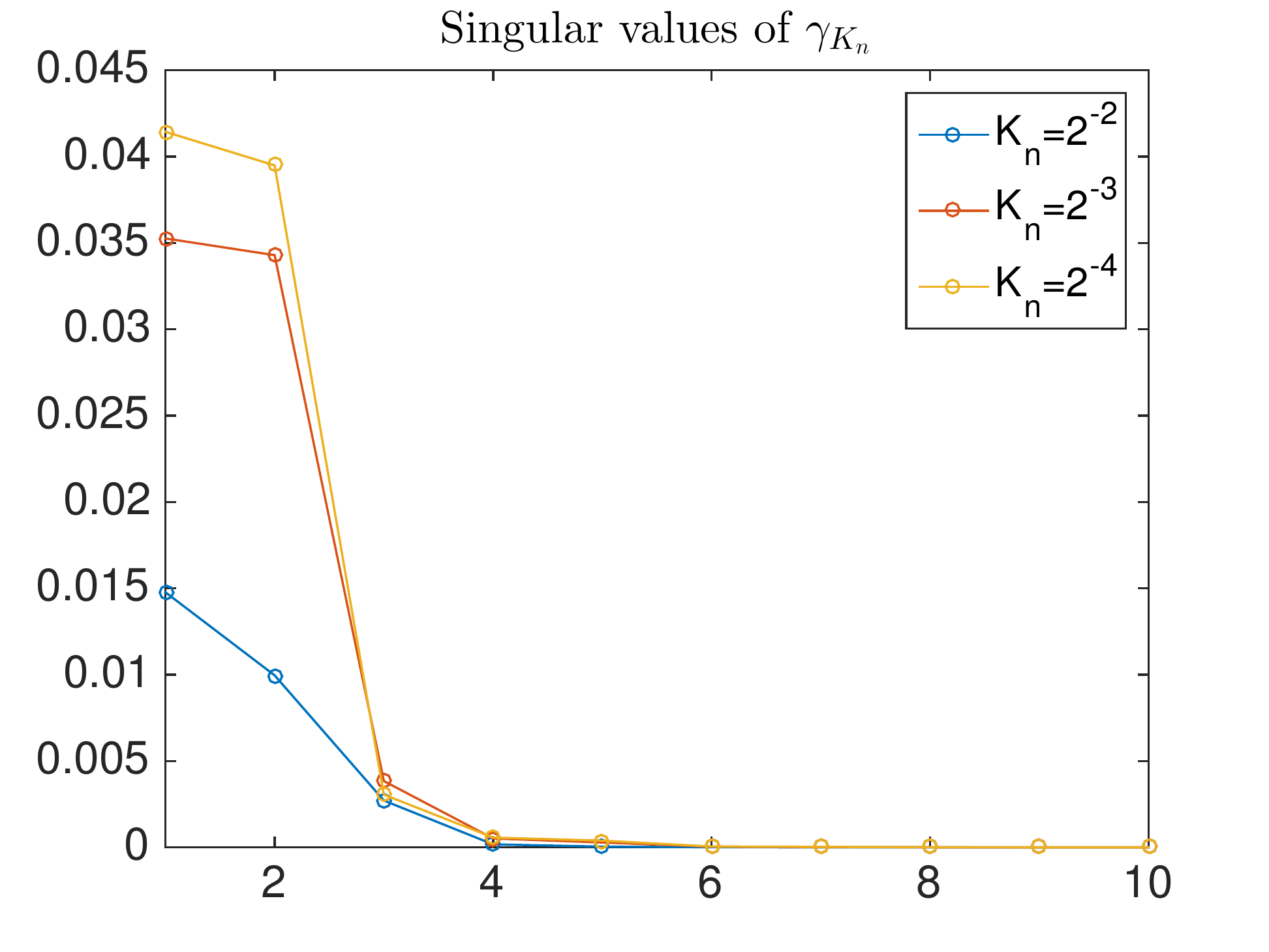}
  \caption{Plots of the singular values for $\gamma_\Kn$ in \eqref{eqn:gamma00} (discrete form is \eqref{Api}) when $\Kn=2^{-2},2^{-3}$ and $2^{-4}$.}
  \label{fig:sin_decay_int_abs}
\end{figure}

Moreover, we see in the proof of Theorem~\ref{thm:Dis} (see equation \eqref{eqn:207}) that $\gamma_\Kn \approx \Kn \rho_f \rho_g$ for small $\Kn$, which lives in a space spanned by $G^2_1$, $G^2_2$ and $G_1G_2$ with $G_{1,2}$ standing for the Green's functions. The plots in~\ref{fig:vectors_int_abs} show the first three eigenvectors of $\Amat$ as $\Kn\to0$.
\begin{figure}[htp]
\centering
  \includegraphics[width=.3\linewidth]{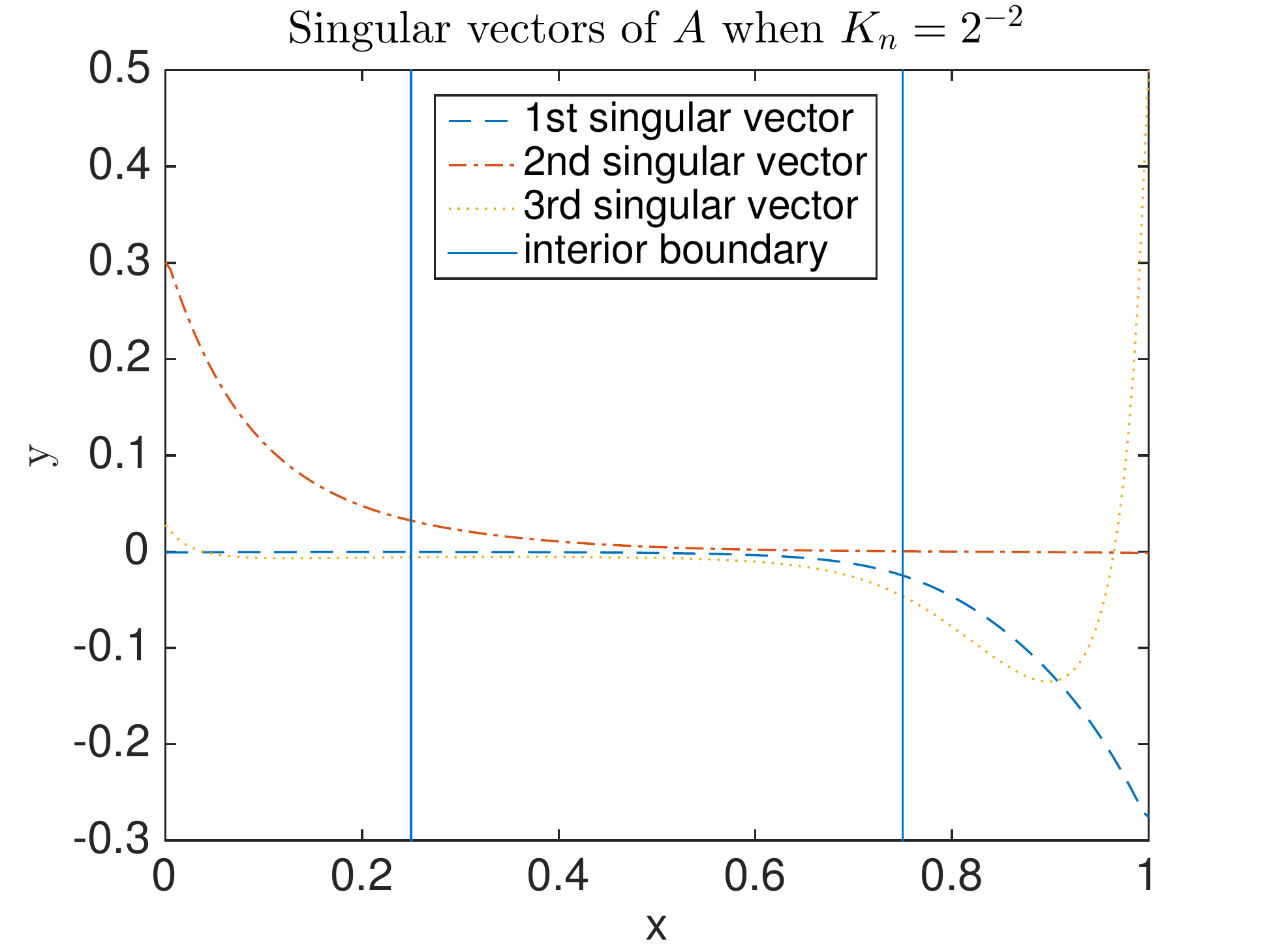}
  \includegraphics[width=.3\linewidth]{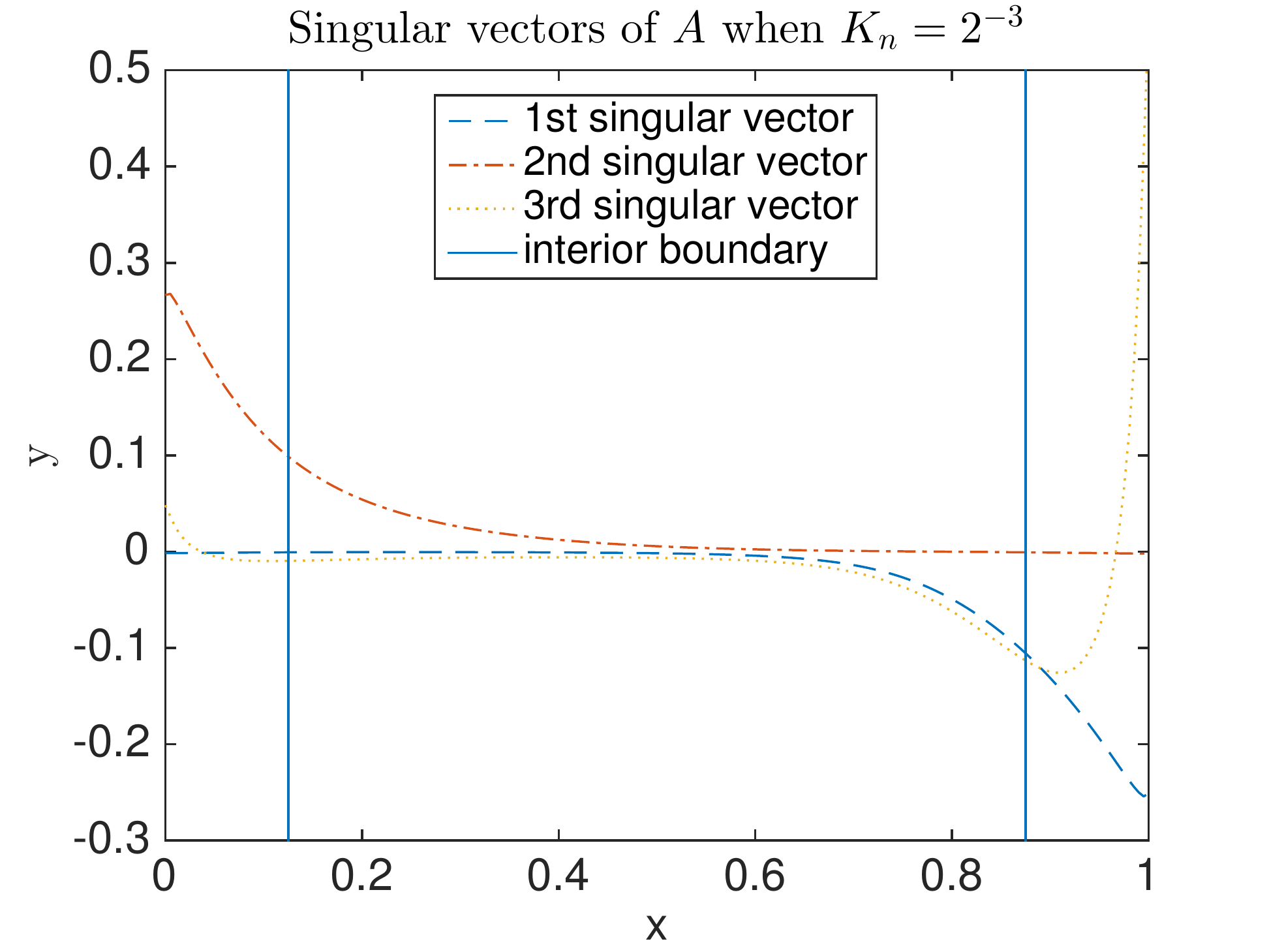}
  \includegraphics[width=.3\linewidth]{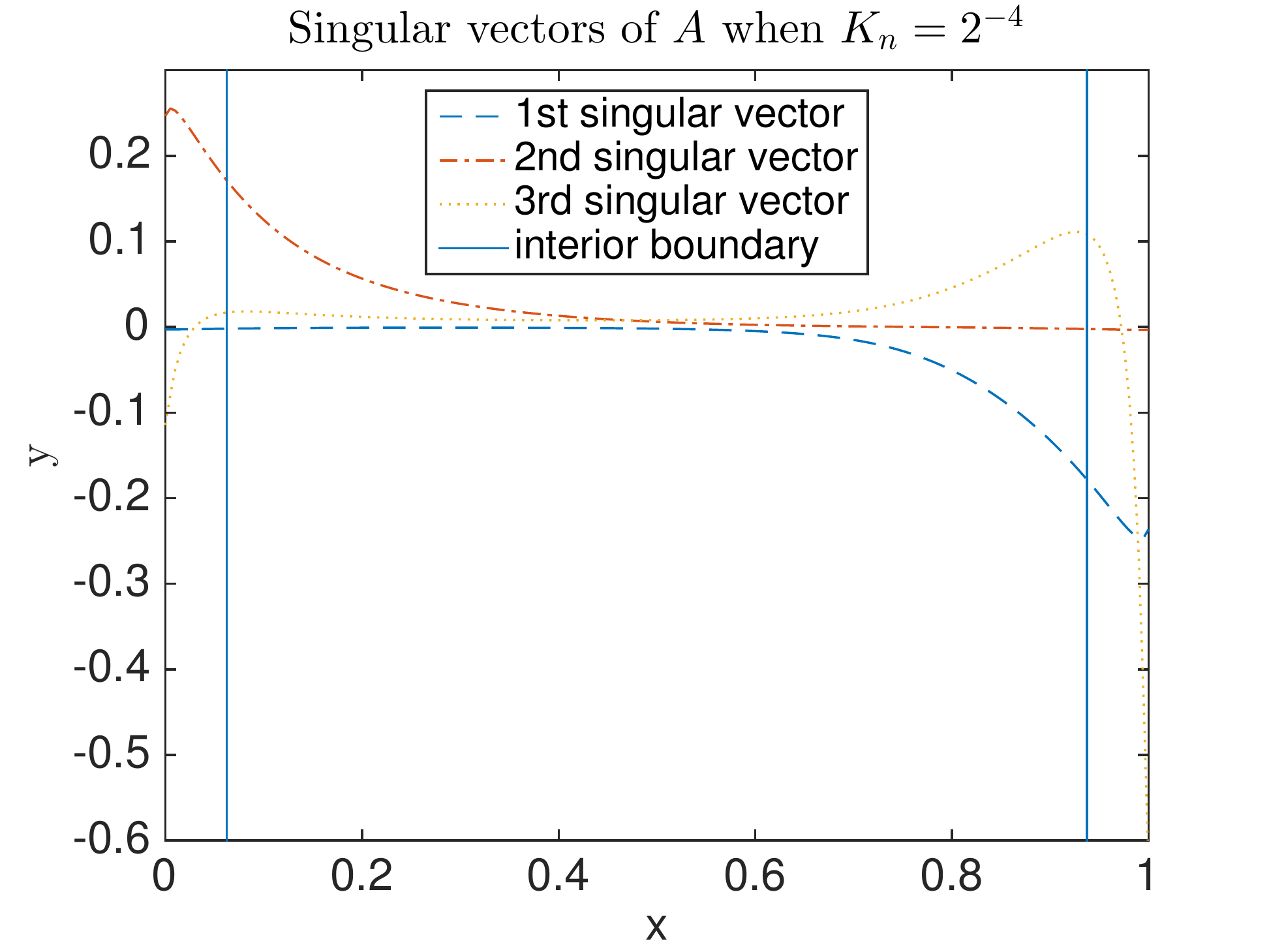}
  \caption{Plots of the singular vectors for $\Amat$ when $\Kn=2^{-2},2^{-3}$ and $2^{-4}$. It can be seen that the first two eigenvectors are almost symmetric to each other. }
  \label{fig:vectors_int_abs}
\end{figure}

\subsection{Recover Scattering Coefficient}
In the second test, we aim to recover $\sigma_s$. Here 
\[
\Amat_{pi}=\frac{1}{\Kn} \int g(x_i,v, \delta_{y_k})\mathcal{L}f_0(x_i,v; \phi_d)\rd v, \quad p = (k,d)\,.
\]
We first compute $\gamma_\Kn$ (stored in $\Amat$) with $\sigma_{s0}=1+\frac{1}{1.5+\sin(2\pi x)}$ and $\sigma_{a}=0$. It is seen from Proposition \ref{thm:RTE_sca} that $\gamma_\Kn$ stays unchanged in $x$ direction, i.e., $\frac{\rd\gamma_\Kn}{\rd x} = 0$. To show this, we plot the matrix $\Amat$ in Fig.~\ref{fig:non-absorb} and see that the entire matrix is roughly flat in $x$ direction.
\begin{figure}[htp]
\centering
  \includegraphics[width=.5\linewidth]{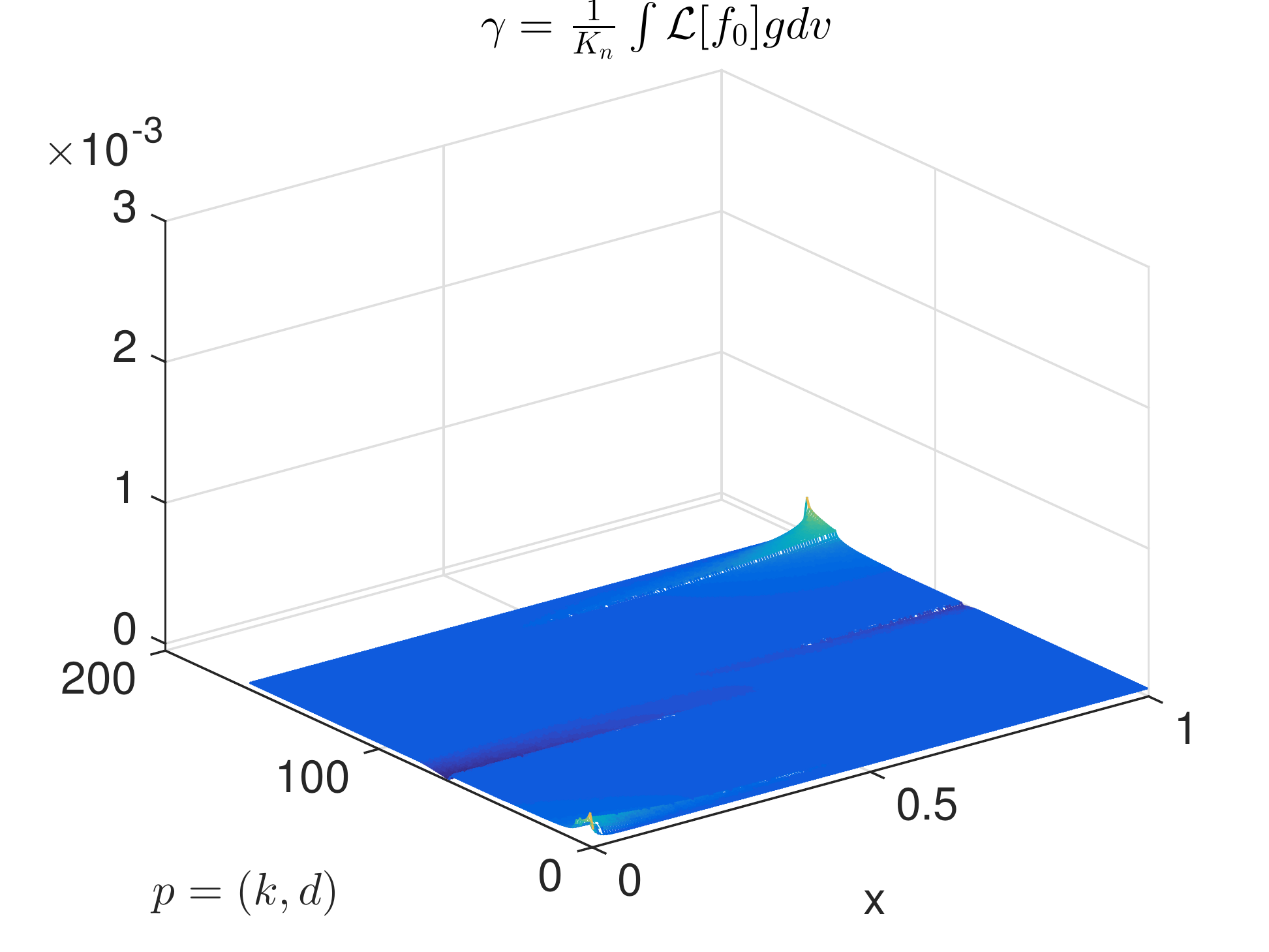}
  \caption{To recover the scattering coefficient when $\sigma_a=0$, the problem is expected to be ill-posed with $\gamma_\Kn$ being flat in $x$. We here plot $\gamma_\Kn = \frac{1}{\Kn}\int \mathcal{L}[f_0]g\rd{v}$ with $\Kn = 1$, for various of $\phi_d$ and $\delta_y$ as the boundary condition for $g$ and $f_0$.}
  \label{fig:non-absorb}
\end{figure}

Next we consider the case when $\sigma_a$ is nontrivial: $\sigma_{a}=2^{-4}(4+\frac{1}{2}\sin(4\pi x))$, and $\sigma_{s0}=\frac{1}{2^{-4}}(1+\frac{1}{1.5+\sin(2\pi x)})$. As predicted in Proposition \ref{prop:RTE_sca2}, $\frac{\rd\gamma_\Kn}{\rd x} \sim \Kn\frac{\rd\rho_f\rho_g}{\rd x}\to 0$ as $\Kn\to 0$, and this is demonstrated in Fig.~\ref{fig:dgamma_ratio}, where we plot $\frac{\rd\gamma_\Kn/\rd x}{\rd\rho_f\rho_g/\rd x}$. As $\Kn$ decreases by $1/2$, the ratio decrease by $1/2$ as well.

\begin{figure}[htp]
\centering
  \includegraphics[width=.5\linewidth]{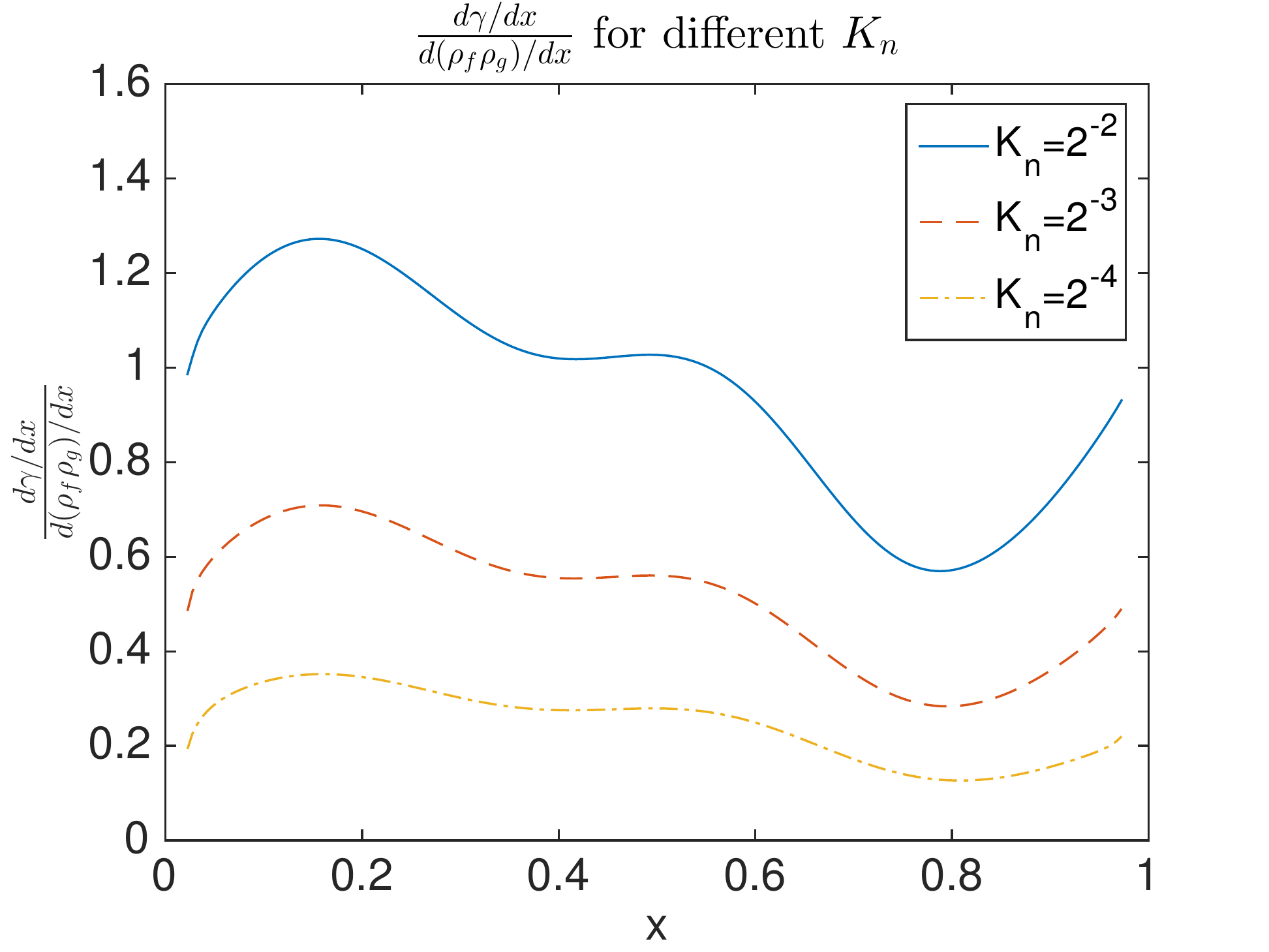}
  \caption{$\frac{\rd\gamma/\rd x}{\rd(\rho_f\rho_g)/\rd x}$. The plot shows that the ratio is indeed at the order $\Kn$ in the zero limit of $\Kn$, as predicted.}
  \label{fig:dgamma_ratio}
\end{figure}

In the end we test the singular value decay of $\Amat_I$, the interior part of $\Amat$. As $\Kn \rightarrow 0$, we expect that its element $\gamma_\Kn \approx \Kn \nabla_x\rho_f\cdot\nabla_x\rho_g$ (see equation \eqref{eqn:0223}), and thus $\Amat_I$ is a low rank matrix. We plot the singular values Fig.~\ref{fig:sin_decay_int_sca}, and as before, only the first three singular values dominates, and the rest vanishes at the order of $\Kn$. 
\begin{figure}[htp]
\centering
  \includegraphics[width=.5\linewidth]{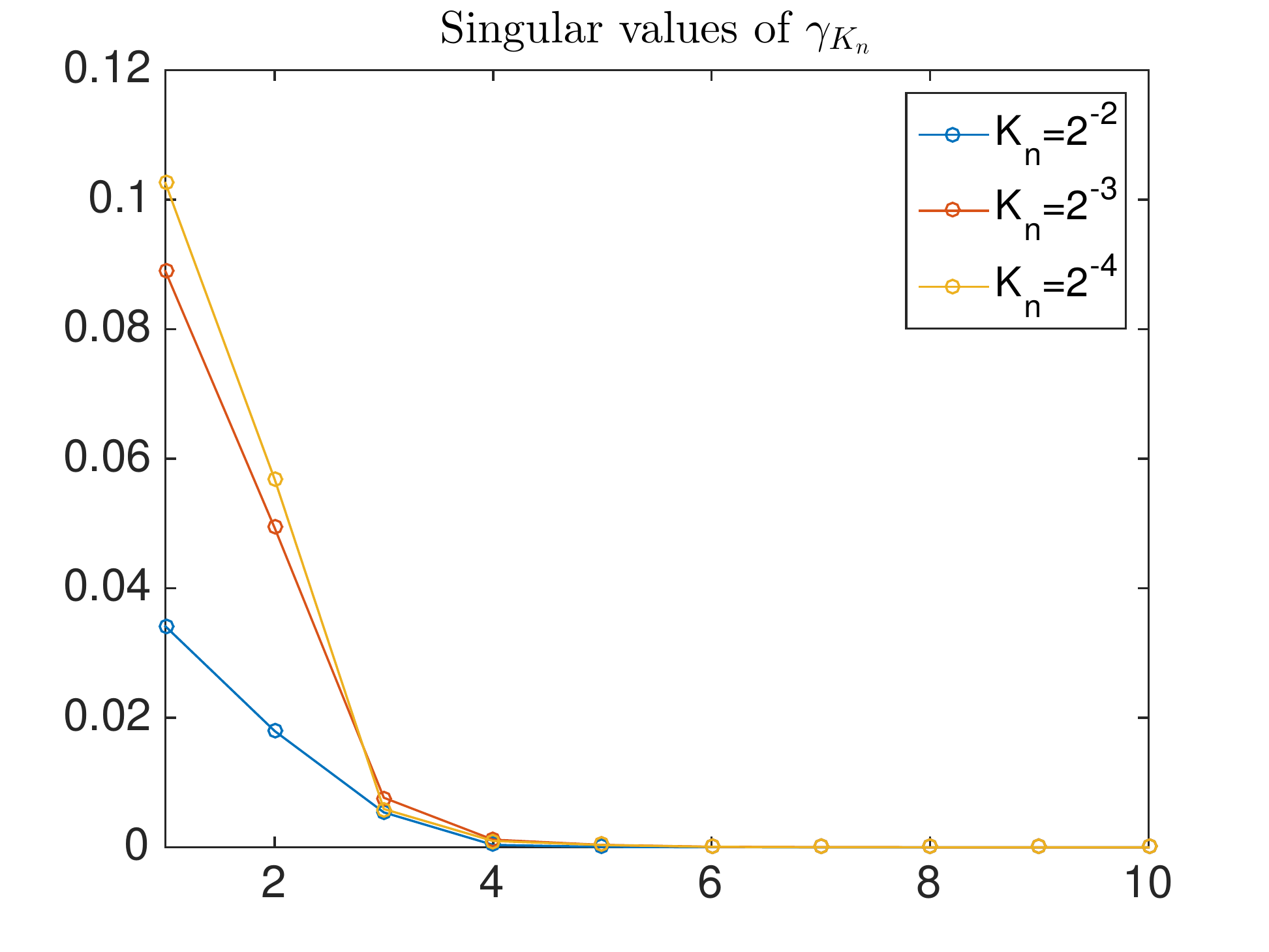}
  \caption{To recover $\sigma_s$, we plot the first few singular values of $\Amat$ in the interior with $\Kn=2^{-2},2^{-3}$ and $2^{-4}$. As $\Kn\to 0$, the problem becomes more and more singular. }
  \label{fig:sin_decay_int_sca}
\end{figure}

We also show that for small $\Kn$, $\gamma_\Kn \approx \Kn \nabla_x\rho_f\cdot\nabla_x\rho_g$, which lives in a space spanned by $(\partial_x G_1)^2$, $(\partial_x G_2)^2$ and $\partial_xG_1\partial_xG_2$ with $G_{1,2}$ standing for the Green's functions. The plots in~\ref{fig:vectors_int_sca} show the first three eigenvectors of $\Amat$ as $\Kn\to0$.
\begin{figure}[htp]
\centering
  \includegraphics[width=.3\linewidth]{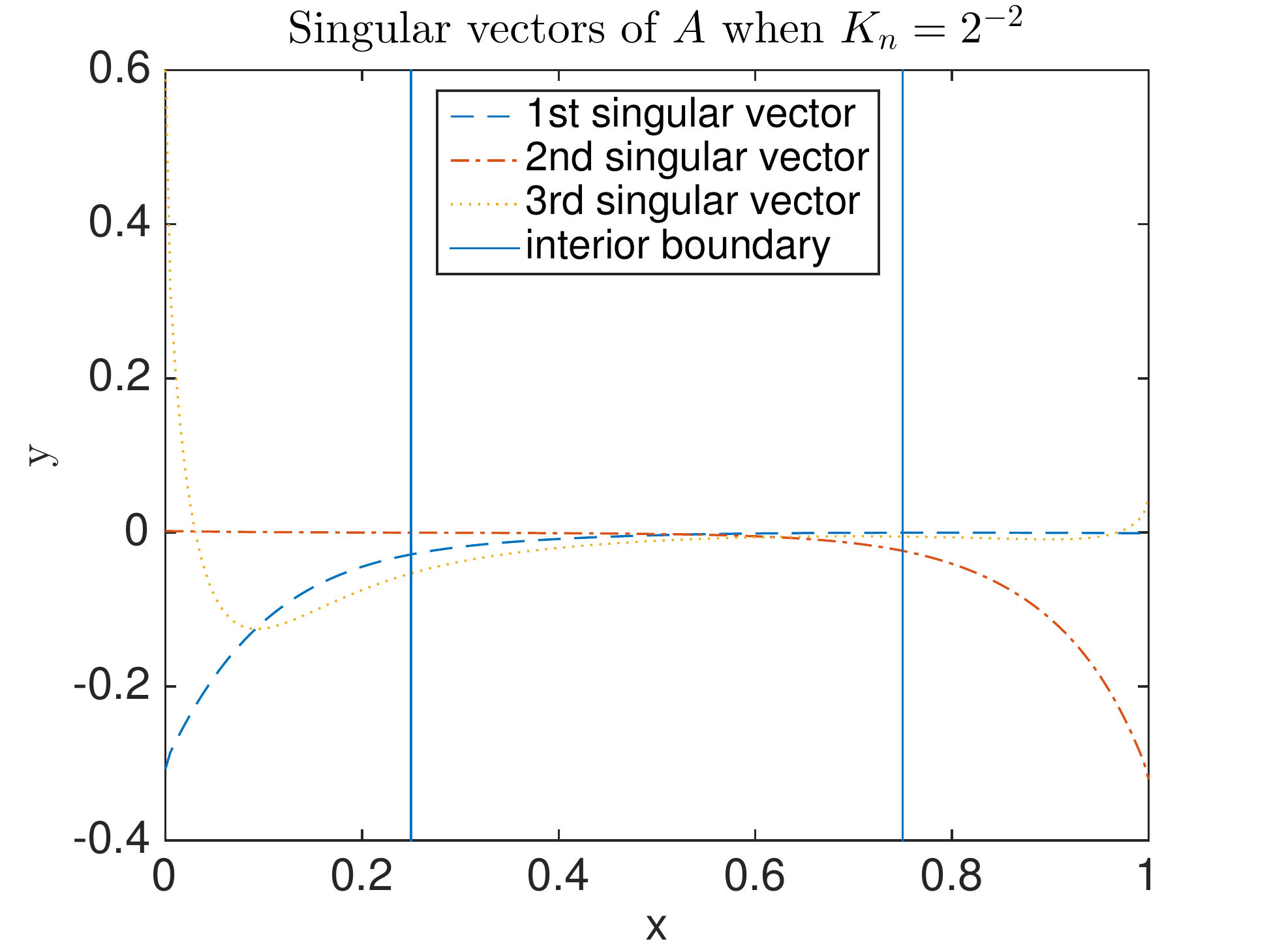}
  \includegraphics[width=.3\linewidth]{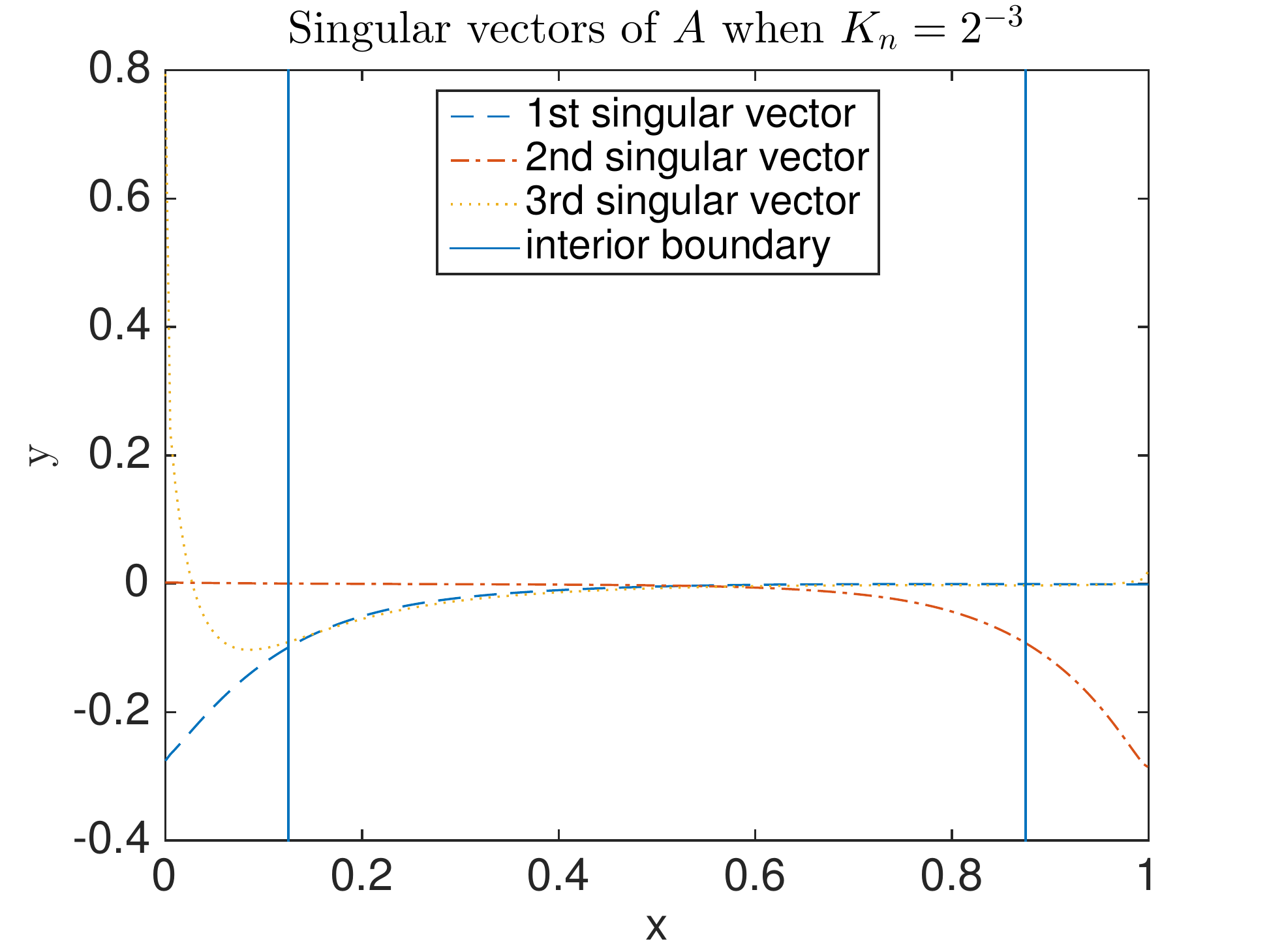}
  \includegraphics[width=.3\linewidth]{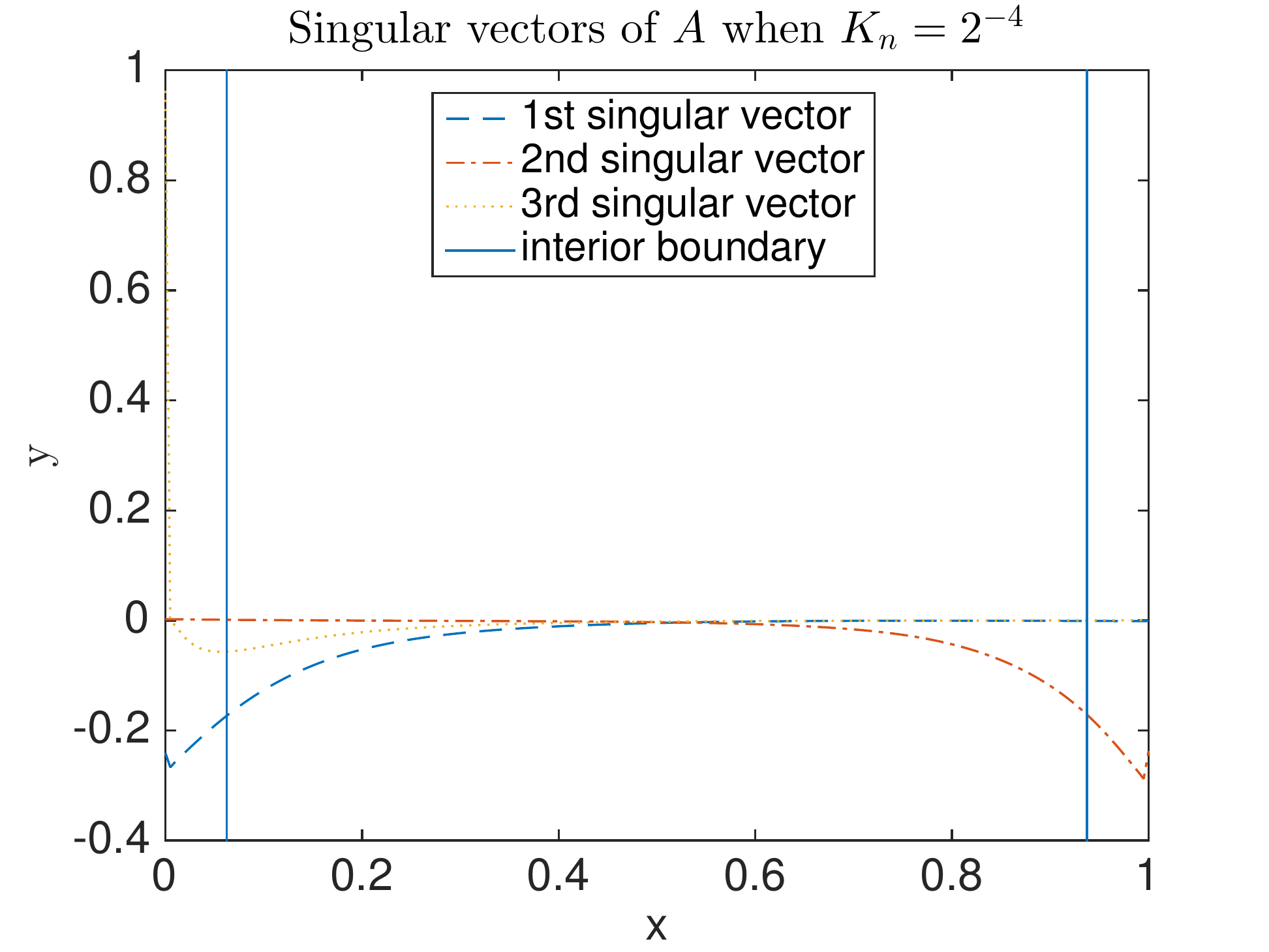}
  \caption{Plots of the singular vectors for $\Amat_I$ when $\Kn=2^{-2},2^{-3}$ and $2^{-4}$. It can be seen that the first two eigenvectors are almost symmetric to each other.}
  \label{fig:vectors_int_sca}
\end{figure}

\section{Appendix}
\hspace{-0.3cm}{\bf Appendix I: Rank of matrix $\Amat_I^0$}
Here we examine the rank of matrix $\Amat_I^0$. We show below that it has rank less than 4 in 1D, but not necessarily low rank in higher dimensions. Denote 
\begin{equation*}
\vec{\rho_f} = (\rho_f(x_1), \rho_f(x_2), \cdots \rho_f(x_{N_x})), \quad 
\vec{\rho_g} = (\rho_g(x_1), \rho_g(x_2), \cdots, \rho_g(x_{N_x}))\,
\end{equation*}
as the solutions to the discrete form of equations \eqref{eqn:rho-f-0} \eqref{eqn:g-0}:
\begin{equation*}
\vec{\rho_f} = \Bmat^{-1} \vec{\xi}, \quad \vec{\rho_g} = \Bmat^{-1} \vec{\delta_y}\,,
\end{equation*}
where $\Bmat_{Nx\times Nx}$ is the discrete version of operator $C\nabla_x \cdot \nabla_x - \sigma_{a0}$. Write $\Bmat^{-1} = (b_{ij})_{N_x \times N_x}$, and recall that the entries of $\Amat_I^0$ is \eqref{eqn:Aij}, we have
\begin{equation*}
\Amat_I^0 = \left(  \begin{array}{cccc}
\left( \sum_{j=1}^{N_x} b_{1j} \xi_j^{(1)} \right) \left( \sum_{j=1}^{N_x} b_{1j} \vec{\delta_{y_j}^{(1)}}\right) w_1
& \cdots
& \cdots
& \left( \sum_{j=1}^{N_x} b_{N_x j} \xi_j^{(1)} \right) \left( \sum_{j=1}^{N_x} b_{N_x j} \vec{\delta_{y_j}^{(1)}} \right) w_{N_x}
\\
\left( \sum_{j=1}^{N_x} b_{1j} \xi_j^{(2)} \right) \left( \sum_{j=1}^{N_x} b_{1j} \vec{\delta_{y_j}^{(2)}}\right) w_1
& \cdots
& \cdots
& \left( \sum_{j=1}^{N_x} b_{N_x j} \xi_j^{(2)} \right) \left( \sum_{j=1}^{N_x} b_{N_x j} \vec{\delta_{y_j}^{(2)}}\right) w_{N_x}
\\ \cdots & \cdots & \cdots & \cdots
\\ \left( \sum_{j=1}^{N_x} b_{1j} \xi_j^{(N_p)} \right) \left( \sum_{j=1}^{N_x} b_{1j} \vec{\delta_{y_j}^{(N_p)}}\right) w_1
& \cdots
& \cdots
& \left( \sum_{j=1}^{N_x} b_{N_x j} \xi_j^{(N_p)} \right) \left( \sum_{j=1}^{N_x} b_{N_x j} \vec{\delta_{y_j}^{(N_p)}}\right) w_{N_x}
\end{array}
\right)\,,
\end{equation*}
Here without abuse of notation, we still the denote number of interior points $N_x$. 

Then in 1D, since we only have two boundary points, $N_y = 2$, and $\xi_j^{(l)} = 0$ for all $j \neq 1$ or $N_x$, and $l= 1, 2, \cdots, N_\phi$. Therefore, $\Amat_I^0$ can be simplified to 
\begin{eqnarray*}
\Amat_I^0  &= & 
\begin{pmatrix} 
\xi_1^{(1)} &    &  &   \\
  & \ddots &   &   & \\
 &  &  \xi_1^{(N_\phi)}  & &  \\
 &  &   &  0  & & \\
  &  &   &    &  \ddots& \\
    &  &   &    &  & 0 \\
\end{pmatrix}
 \begin{pmatrix}
b_{11}^2   &   b_{21}^2    &   \cdots  & \cdots  & b_{N_x1}^2 \\
 \vdots      &   \vdots        &    \vdots  &  \vdots  &   \vdots  \\
b_{11}^2  &     b_{21}^2  &     \cdots  & \cdots  &   b_{N_x1}^2 \\
&  &  &  & \\
 &  &  &  & \\
 &  &0  &  & \\
 &  &    &  &
\end{pmatrix} 
\begin{pmatrix}
w_1 &   &   \\
  & \ddots &  \\
&  & w_{N_x} \\
\end{pmatrix}  
\\
&+ &
\begin{pmatrix} 
\xi_{N_x}^{(1)} &    &  &   \\
  & \ddots &   &   & \\
 &  &  \xi_{N_x}^{(N_\phi)}  & &  \\
 &  &   &  \xi_1^{(N_\phi+1)}  & & \\
  &  &   &    &  \ddots& \\
    &  &   &    &  & \xi_1^{(2N_\phi)} \\
\end{pmatrix}
 \begin{pmatrix}
b_{11} b_{1N_x}   &   b_{21}b_{2N_x}    &   \cdots  & \cdots  & b_{N_x1}b_{N_xN_x} \\
 \vdots      &   \vdots        &    \vdots  &  \vdots  &   \vdots  \\
b_{11} b_{1N_x}  &      b_{21}b_{2N_x}    &   \cdots  & \cdots  & b_{N_x1}b_{N_xN_x}\\
b_{11} b_{1N_x}   &   b_{21}b_{2N_x}    &   \cdots  & \cdots  & b_{N_x1}b_{N_xN_x} \\
 \vdots      &   \vdots        &    \vdots  &  \vdots  &   \vdots  \\
b_{11} b_{1N_x}  &      b_{21}b_{2N_x}    &   \cdots  & \cdots  & b_{N_x1}b_{N_xN_x}
\end{pmatrix} 
\begin{pmatrix}
w_1 &   &   \\
  & \ddots &  \\
&  & w_{N_x} \\
\end{pmatrix} 
\\
&+&
\begin{pmatrix} 
0 &    &  &   \\
  & \ddots &   &   & \\
 &  &  0  & &  \\
 &  &   &  \xi_1^{(N_\phi+1)}  & & \\
  &  &   &    &  \ddots& \\
    &  &   &    &  & \xi_1^{(2N_\phi)} \\
\end{pmatrix}
 \begin{pmatrix}
&  &  &  & \\
 &  &  &  & \\
 &  &0  &  & \\
 &  &    &  & \\
 b_{11}^2   &   b_{21}^2    &   \cdots  & \cdots  & b_{N_x1}^2 \\
 \vdots      &   \vdots        &    \vdots  &  \vdots  &   \vdots  \\
b_{11}^2  &     b_{21}^2  &     \cdots  & \cdots  &   b_{N_x1}^2 \\
\end{pmatrix} 
\begin{pmatrix}
w_1 &   &   \\
  & \ddots &  \\
&  & w_{N_x} \\
\end{pmatrix}   \,.
\end{eqnarray*}
Then it is easy to see that each component in the summation is rank one, therefore in total the rank of $\Amat_I^0$ does not exceed $3$.

In higher dimensions, on the contrary, does not guarantee a low rank property of $\Amat_I^0$. To see this, let $N_J = \{j_0, j_1, \cdots, j_J\} \subset \{ 1, 2, \cdots, N_x \}$ be the collection of indices such that $\xi_j^{l} \neq 0$, $j \in N_J$, $\forall l $, then $J< N_x$. Without loss of generality, we let $\vec{\delta_{y}^{(1)}}= \cdots =  \vec{\delta_{y}^{(N_\phi)}} = \vec{\delta_{y_{j_1}}}$, $\vec{\delta_{y}^{(N_\phi+1)}} = \cdots =  \vec{\delta_{y_{j_2}}}$, $\cdots$, $\vec{\delta_{y}}^{(Ny-1)N_\phi+1} = \cdots = \vec{\delta_{y_{j_{N_y}}}}$, then
\begin{equation*}
\Amat_I^0 =    \Amat_1 * \left( \begin{array}{ccc} b_{1j_1} & & \\ & \cdots & \\ & & b_{N_x j_1}  \end{array}\right) + 
\Amat_2 * \left( \begin{array}{ccc} b_{1j_2} & & \\ & \cdots & \\ & & b_{N_x j_2}  \end{array}\right) + 
\cdots + 
\Amat_{N_y} * \left( \begin{array}{ccc} b_{1j_{N_y}} & & \\ & \cdots & \\ & & b_{N_x j_{N_y}}  \end{array}\right)
\,,
\end{equation*}
where
\begin{equation*}
\Amat_1 = \left( \begin{array}{ccc}
& \AAmat & 
\\ & &
\\ & & 
\\ & \vec{0} & 
\\ & & 
\\ & & 
\end{array}
\right)_{N_p \times N_x}, \qquad
\Amat_2 = \left( \begin{array}{ccc}
& 0 & 
\\ & \AAmat & 
\\ & & 
\\ & \vec{0} & 
\\ & & 
\\ & & 
\end{array}
\right)_{N_p \times N_x}, \qquad  \cdots
\Amat_{N_y} = \left( \begin{array}{ccc}
&  & 
\\ && 
\\ &\vec{0} & 
\\ &  & 
\\ & & 
\\ & \AAmat & 
\end{array}
\right)_{N_p \times N_x}
\end{equation*}
and
\begin{equation*}
\AAmat = \left( \begin{array}{cccc} 
\sum_{b_{1j}} \xi_{j}^{(1)}  & \sum_{b_{2j}} \xi_{j}^{(1)} & \cdots & \sum_{b_{N_xj}} \xi_{j}^{(1)}\\
\vdots & \vdots & \vdots & \vdots
\\ 
\sum_{b_{1j}} \xi_{j}^{(N_\phi)}  & \sum_{b_{2j}} \xi_{j}^{(N_\phi)} & \cdots & \sum_{b_{N_xj}} \xi_{j}^{(N_\phi)} 
\end{array} \right)_{N_\phi \times N_x}\,.
\end{equation*}
Therefore, although $\text{rank}(\AAmat) \leq N_J$, adding them together may still give a full rank matrix $\Amat_I^0$.

\bibliographystyle{siam}
\bibliography{inverse_RTE_ref}

\end{document}